\documentclass{amsart}
\usepackage{amssymb}
\usepackage{graphicx}
\usepackage[usenames]{color}

\usepackage[colorlinks = true,
        urlcolor  = blue]{hyperref}
        \usepackage{xcolor}
\usepackage{cancel}

\usepackage{multirow} 
\usepackage{algorithm}
\usepackage{algpseudocode}

\newtheorem{theorem}{Theorem}[section]
\newtheorem{proposition}[theorem]{Proposition}

\newtheorem{assumption}{Assumption}
\newtheorem{lemma}[theorem]{Lemma}
\newtheorem{remark}[theorem]{Remark}
\newtheorem{definition}[theorem]{Definition}
\newtheorem{example}[theorem]{Example}

\usepackage{xcolor}

\usepackage{subcaption}

\begin{document}
\title[Conditional Propagation of Chaos]{Emergence of Common Noise: Quantitative Conditional Propagation of Chaos}
\author{Vlad Bally,  Eva L\"ocherbach}
\address{V. Bally: LAMA, Universit\'e Gustave Eiffel, France.\\
Email vlad.bally@univ-eiffel.fr}
\address{E. L\"ocherbach: CMAP, 
Ecole Polytechnique, France.\\
E-mail: eva.loecherbach@polytechnique.edu}

\begin{abstract}
We study a dynamical invariance principle for interacting particle systems
with mean-field interactions and common noise emerging through a collective
stochastic perturbation. The particle dynamics combine autonomous evolution
with a weakly scaled random bombardment whose cumulative effect generates a
Brownian common noise in the large population limit. Working in a general
abstract framework based on stochastic flows and the stochastic sewing
lemma, we establish quantitative conditional propagation of chaos estimates
for both discrete Euler schemes and their continuous-time flow limits. Our
approach yields explicit Wasserstein convergence rates and applies in
particular to jump-diffusion models motivated by interacting spiking neuron
systems. The analysis relies on quantitative central limit theorems of Rio
and Bonis together with a stochastic sewing argument adapted to $L^2$-valued
flows.
\end{abstract} 

\maketitle
\textbf{Keywords}: quantitative conditional propagation of chaos,  emergence of common noise, central limit theorem, 
stochastic sewing lemma.

\textbf{AMS Classification 2020:} 60F17, 60F15

\section{Introduction}

\subsection{Emergence of common noise from collective perturbations}

The purpose of this paper is to study the emergence of common noise in
interacting particle systems and to obtain quantitative conditional
propagation of chaos estimates for the associated mean-field limits.

We work in a fairly general framework with
discrete time schemes constituted by $N$ particles taking values in $\mathbb R^d$.  
Each particle is subject
to two sources of randomness, an autonomous part and a stochastic
perturbation term generated by the aggregation of many small stochastic
inputs to which each particle is equally exposed. This common
stochastic perturbation term leads to the presence of a common source of
noise $W$ in the mean field limit.

The main example we have in mind is when this common noise $W$ is a standard Brownian motion 
and when each particle follows an equation 
of jump-diffusion type, given by
\begin{eqnarray}
\mathbb{X}_{s,t}^{N,i} &=&\mathbb{X}^{N,i}+\int_{s}^{t}\sigma _{0}(\mathbb{X}%
_{s,r}^{N,i},\mu _{N}( \mathbb{X}_{s,r}^{N}, \varphi))dr+\int_{s}^{t}\sigma
_{1}(\mathbb{X}_{s,r}^{N,i},\mu _{N}(\mathbb{X}%
_{s,r}^{N}, \varphi)dB_{r}^{i}  \notag \\
&&+\int_{]s,t]}\int_{E_{1}}c(u,\mathbb{X}_{s,r}^{N,i},\mu _{N}(%
\mathbb{X}_{s,r}^{N}, \varphi))N^{i}(dr,du)  \notag \\
&&+\frac{1}{\sqrt{N}}\left( \sum_{j=1}^{N}\int_{s}^{t}\overline{\sigma }(%
\mathbb{X}_{s,r}^{N,j})d\overline{B}_{r}^{j}+\sum_{j=1}^{N}\int_{]s,t]}%
\int_{E_{2}}\overline{c}(u,\mathbb{X}_{s,r}^{N,j}))\widetilde{M}%
^{j}(dr,du)\right) ,  \label{EX1intro}
\end{eqnarray}%
for any $0\leq s\leq t<\infty ,$ and any $i=1,...,N.$ 

In the above system, $B^{i},\overline{B}^{i},i=1,\ldots ,N,$ are independent
standard Brownian motions, and $N^{i}$ and $M^{i},i=1,\ldots ,N,$ are
independent Poisson random measures on some measurable space $(E_{1},\nu
_{1})$ ($(E_{2},\nu _{2}),$ respectively), with intensity $dr\nu _{1}(du)$ ($%
dr\nu _{2}(du),$ respectively). Moreover, for each $1\leq i\leq N,$ $%
\widetilde{M}^{i}(dr,du)=M^{i}(dr,du)-dr\nu _{2}(du)$ denotes the
compensated Poisson random measure, and $\mu _{N}(\mathbb{X}%
_{s,r}^{N},\varphi)=\frac{1}{N}\sum_{i=1}^{N}\varphi (\mathbb{X}_{s,r}^{N,i})$.
In this example, the collective random perturbation is given by the terms of the last line in (\ref{EX1intro}).

Each individual input of the collective random perturbation is negligible, but their cumulative
effect becomes visible at the macroscopic level. After the natural
central-limit scaling $N^{-1/2}$, these perturbations generate a
Brownian common noise in the large population limit.

As a consequence of the presence of this Brownian common noise, the limiting particles are no longer independent.
Instead, they are independent conditionally on the realization of the
common noise. The limiting equation is therefore a conditional
McKean--Vlasov equation and the relevant notion of convergence is the
conditional propagation of chaos property. 

Coming back to the system of (\ref{EX1intro}), its associated infinite limit system
is given by 
\begin{eqnarray}
\widehat{\mathbb{X}}_{s,t}^{i} &=&\widehat{\mathbb{X}}^{i}+\int_{s}^{t}%
\sigma _{0}(\widehat{\mathbb{X}}_{s,r}^{i},\widehat{\mathbb{\mu }}_{s, r }(\varphi
))dr+\int_{s}^{t}\sigma _{1}(\mathbb{X}%
_{s,r}^{i},\widehat{\mu }_{s, r }(\varphi ))dB_{r}^{i}  \notag \\
&&+\int_{]s,t]}\int_{E_{1}}c(u,\widehat{\mathbb{X}}_{s,r}^{i},\widehat{\mu 
}_{s, r}(\varphi ))N^{i}(dr,du)  
+\int_{s}^{t}\sqrt{\widehat{\mu }_{s, r}(f)}%
dW_{r} . \label{eq:limit}
\end{eqnarray}%
Here $\widehat{\mu }_{s, r}(\varphi)$ is the
conditional expectation $E(\varphi (\widehat{\mathbb{X}}_{s,r}^{i})| 
\mathcal{F}_s^W  ),$ where we put $\mathcal{F}_{r,s}^W = \sigma ( {\ W_u ,r\le  u \le s }),$ and the function $f$ is related to the variance of the input
given by the last line of equation (\ref{EX1intro}). 

A motivating example for our work is given by systems of interacting
spiking neurons arising in mathematical neuroscience, see \cite{AAEE} and 
\cite{FL}. In these models,
each neuron receives a large number of small random synaptic inputs.
It was shown in \cite{ELL} that although the influence of any individual input vanishes as the
population size grows, the cumulative effect of all inputs survives
and generates a common Brownian motion in the limit, 
leading to conditional propagation of chaos. The approach of \cite{ELL} was
based on the strong coupling result of Koml\'{o}s, Major and Tusn\'{a}dy \cite{KMT} that allows to couple a
centered Poisson process with a standard Brownian motion (KMT coupling).
Applying this strong coupling on small time intervals and working with a
discrete time approximation scheme of the original particle system, 
an explicit rate of convergence for a modified Wasserstein $1-$%
distance was obtained in \cite{ELL}, and  this rate was shown to be  of order $N^{-1/10}.$ 

\subsection{Main contribution}
The objective of the present paper is to develop a general framework
for quantitative conditional propagation of chaos in situations where
the common noise emerges through a dynamical invariance principle.

Our approach differs substantially from the previous work based on
strong KMT couplings. The key observation is that, in many situations,
the emergence of the common noise can be viewed as a quantitative
central-limit phenomenon. We therefore replace KMT arguments by
quantitative Wasserstein estimates in a general setting. If this setting is given
by the central limit theorem, then we rely on the quantitative estimates obtained by
Rio \cite{Rio} in dimension one and by Bonis \cite{Bonis} in arbitrary
dimension.

This point of view leads to two main contributions.

First, we introduce an abstract framework for interacting particle
systems subject to collective perturbations. Within this framework,
the approximation of the original perturbation by its Gaussian limit
is controlled by Wasserstein estimates at the level of a single
time step.

Second, combining these quantitative CLT estimates with a stochastic
sewing argument, we obtain quantitative propagation-of-chaos bounds
both for discrete Euler schemes and for the associated continuous-time
stochastic flows.

In the central-limit setting, the resulting
error is of order $N^{-1}$ for the squared Wasserstein distance,
corresponding to the optimal central-limit scaling.

\subsection{Proof strategy}
The proof relies on a three-level approximation procedure.

Starting from the original particle system $X^N$, we first introduce an
auxiliary  system $\bar X^N$ obtained by replacing the
collective perturbation by its Gaussian limit. Quantitative
Wasserstein bounds in the central limit theorem allow us to control
the distance between these two systems. 
This auxiliary system $ \bar X^N$ can 
be seen as a mean-field approximation of the limit system
in which the limit driving common noise $W$ is already present.

We then compare $\bar X^N$ with a conditionally independent system
$\hat X^N$ driven by the same common noise. This second step relies on
conditional independence and quantitative law-of-large-numbers
estimates.

Schematically, we follow the approximation scheme 

\[
X^N
\longrightarrow
\bar X^N
\longrightarrow
\hat X^N .
\]

The first arrow is governed by the quantitative central limit theorem,
while the second arrow corresponds to the law of large numbers, conditionally 
on the source of common noise. 

The analysis is first carried out for Euler schemes on dyadic
partitions. We then pass to continuous time by means of a stochastic
sewing argument in an $L^2$ setting. This yields convergence estimates
for the associated stochastic flows.

\subsection{Main result}

Our main theorem provides precise error bounds for the conditional
propagation-of-chaos property  both for  Euler schemes and
for the associated continuous-time flows. We write 
$X_{s,t }^{n,N}(X)$ for the Euler scheme associated with $X^N,$ with dyadic step size $ 2^{-n }, $
starting from the initial configuration $ X = ( X^1, \ldots, X^N )$ at time $s$ and evolving
up to time $t.$ Similarly, we write $\widehat{X}_{s,t}^{n,N}(X) $ for the 
Euler scheme associated with the limit system $ \hat X^N.$ Then our main result reads as follows.

\begin{theorem}[Informal statement]
Assume that the Gaussian approximation of the collective perturbation
satisfies a one-step Wasserstein estimate of order $\bar\varepsilon_N$
and that the associated conditional law of large numbers error is of
order $\hat\varepsilon_N$.

Under convenient Lipschitz and growth assumptions, if
$X^1,\ldots, X^N$ are i.i.d. (conditionally on the common noise when
present), then there exists a constant $C$, independent of $N$ and of
the discretization parameter $ n,$ such that

\[
\mathcal{W}_{2}^{2}(X_{s,t}^{n,N}(X),\widehat{X}_{s,t}^{n,N}(X))
\le
C\Big(
\frac1N+\bar\varepsilon_N+\hat\varepsilon_N
\Big).
\]
The same estimate holds for the associated continuous-time flows.
In the central-limit setting, the quantitative
results of Rio and Bonis imply

\[
\bar\varepsilon_N+\hat\varepsilon_N
=
O(N^{-1}),
\]
so that
\[
\mathcal{W}_{2}^{2}(X_{s,t}^{n,N}(X),\widehat{X}_{s,t}^{n,N}(X))
\le
\frac{C}{N} .
\]
\end{theorem}
Here, $\mathcal{W}_{2} $ denotes the Wasserstein $2-$distance on ${\mathbb{R}%
}^{d \times N },$ and we identify $\mathcal{W}_{2} (X, Y) = 
\mathcal{W}_{2} ( {\mathcal{L}}(X), {\mathcal{L}} (Y) ).$
We refer to Theorem \ref{theo:main} for the precise statement. The proof of the theorem combines a quantitative central limit theorem, a conditional law of large numbers, and a stochastic sewing argument.

Let us finally give some comments on the use of the stochastic sewing lemma in our paper.
The sewing lemma has been introduced in the early 2000's by Gubinelli \cite%
{Gubinelli} and by De la Pradelle and Feyel \cite{arnauddenis}. It has shed
new insight on the theory of rough paths, see in particular Davie \cite%
{davie} and the monograph by Friz and Hairer \cite{frizhairer}. A long list
of variants and generalizations have followed since then. The present paper
relies on the new results of L\^e \cite{Le} on the so-called stochastic
sewing lemma. A variant of this result for flows on the Wasserstein space
has been recently obtained by Alfonsi et al \cite{ABC}. However, here we are
concerned with applications on the $L^{2}-$space, instead of working on the
Wasserstein space. Therefore we have to adapt the stochastic sewing lemma to
this framework. This is done in Section \ref{sec:2} below.

\subsection{Organization of the paper}

Our paper is organized as follows. In Section 2 we state the stochastic
sewing lemma on the $L^2-$space in Lemma \ref{sewing}. In Section \ref{sec:4}
we introduce the core objects of the present paper and define the discrete
time dynamics of the stochastic particle systems. Here we also state our
main result, Theorem \ref{theo:main}, that gives the precise speed of
convergence for the Wasserstein $2-$distance both for the Euler schemes and
the limit flows in continuous time. 

The proof of Theorem \ref{theo:main} is given in Section \ref{sec:4true}.
Here we state in Lemma \ref{lem:44} the main coupling result between the one step 
Euler schemes for $ X^N$ and $ \bar X^N.$ The proof
of Theorem \ref{theo:main} is then completed in Section \ref{LN8} where we
control the error when passing from the auxiliary scheme $\bar X^N$ to the limit one $
\hat X^N .$ The main
arguments rely on the conditional independence of the coordinates of the
limit scheme, conditionally on $W,$ such that we may apply the strong law of
large numbers.

We then discuss in Section \ref{sec:CLT} the main example of our method
where the passage from $X^N $ to $\bar X^N $ is governed by
the Central Limit Theorem, using the
results of E. Rio and T. Bonis (\cite{Rio} and \cite{Bonis}). Finally, in
Section \ref{sec:5}, we apply our results to a general version of (\ref%
{EX1intro}).

\subsection{Notation}

Throughout this paper we will use the following notation.

\begin{itemize}
\item We note $C$ any arbitrary constant. The value of $C$ can change from
occurrence to occurrence (even within the same line) in an equation.
Moreover, if $C$ depends on some non-fixed parameter $\theta$, we write $%
C_\theta$.

\item For any $d \geq 1 $ and $x \in {\mathbb{R}}^d, $ we will write $|x| $
for the Euclidean norm of $x.$

\item For fixed $d \geq 1$ and any $d \times d-$matrix $A,$ we will use the
Hilbert-Schmidt norm $\| A\|_{HS} = \sqrt{\sum_{i, k = 1}^d |A^{i,k}|^2 }.$

\item For fixed $d \geq 1$ and any $d \times d-$matrix $A,$ the operator
norm is defined by $\| A\|_{op} = \sup_{ x \in {\mathbb{R}}^d , x \neq 0} 
\frac{ | Ax |}{|x|}.$ We will use that $\|A\|_{op} = \sqrt{ \lambda_{max} (
A A^* ) } , $ where $\lambda_{max} $ denotes the largest eigenvalue of a
matrix.

\item We also fix $d_{0}\leq d$ and for $w=(w^{1},...,w^{d_{0}})\in {\mathbb{%
R}}^{d_{0}} $ we define $1_{d_{0}}w\in {\mathbb{R}}^{d}$ by $1_{d_{0}}w=(0,
\ldots, 0, w^{1},...,w^{d_{0}}).$

\item If $\varphi : E \to {\mathbb{R}}^m $ is bounded, then we write $\|
\varphi \|_\infty = \sup_{ x \in E } | \varphi( x) | .$

\item $L_\varphi $ denotes the Lipschitz constant of a Lipschitz continuous
function $\varphi : E \to E^{\prime }$ defined on some metric space $( E, d) 
$ and taking values in another metric space $(E^{\prime }, d^{\prime }), $
that is, 
\begin{equation*}
L_\varphi = \sup_{ x , y \in E, x \neq y} \frac{ d^{\prime }( \varphi ( x),
\varphi (y ) )}{ d ( x,y) }.
\end{equation*}

\item We will work with a fixed filtration $( \mathcal{F}_t)_{t \geq 0}, $
and $E_s ( X)$ denotes the conditional expectation of $X$ with respect to $%
\mathcal{F}_{s},$ and $\tilde E_s( X) = X - E_s (X).$

\item We denote by $L_{t}^{2}({\mathbb{R}}^{d })$\textbf{\ }the space of $%
\mathcal{F}_{t}-$measurable, square integrable ${\mathbb{R}}^{d }-$valued
random variables, and by $L_{t}^{2}({\mathbb{R}}^{d\times N})$ the space of $%
\mathcal{F}_{t}-$measurable, square integrable ${\mathbb{R}}^{d\times N}-$%
valued random variables.

\item For any Polish space $(E, \mathcal{E}),$ we denote by $\mathcal{P}( E) 
$ the space of all probability measures on $( E, \mathcal{E}).$
\end{itemize}

\section{\protect\bigskip The stochastic sewing lemma}

\label{sec:2}

We consider a probability space $(\Omega ,\mathcal{F},P)$ with a filtration $%
(\mathcal{F}_{t})_{t \geq 0}$ and a Hilbert space $H$ (which will be ${%
\mathbb{R}}^{d}$ or ${\mathbb{R}}^{d\times N}$ in the following), equipped
with its norm $| \cdot | .$ We denote by $L_{s}^{2}(H)$ the space of all
square integrable $H-$valued random variables $X$ which are $\mathcal{F}%
_{s}- $ measurable. We shall write $\|X\|_2^2 = E |X|^2 $ for the associated
square of the $L^2-$norm.

The main object in this section is a family of applications $\mathcal{X}%
_{s,t}:L_{s}^{2}(H)\rightarrow L_{t}^{2}(H),0\leq s<t < \infty, $ such that $%
\mathcal{X}_{s,s}(X)=X.$ We also denote%
\begin{eqnarray}  \label{sf-1}
\mathcal{E}_{s,t}^{\mathcal{X}}(X) &=&\mathcal{X}_{s,t}(X)-X, \\
\widehat{\mathcal{E}}_{s,t}^{\mathcal{X}}(X) &=&E_{s}(\mathcal{E}_{s,t}^{%
\mathcal{X}}(X)), \quad \widetilde{\mathcal{E}}_{s,t}^{\mathcal{X}}(X)=%
\mathcal{E}_{s,t}^{\mathcal{X}}(X)-\widehat{\mathcal{E}}_{s,t}^{\mathcal{X}%
}(X),  \notag
\end{eqnarray}%
where $E_{s}$ designs the conditional expectation with respect to $\mathcal{F%
}_{s}.$ We think of $\widehat{\mathcal{E}}_{s,t}^{\mathcal{X}}(X) $ as being
the \textit{finite-variation} part of $\mathcal{E}_{s,t}^{\mathcal{X}}(X)$
and of $\widetilde{\mathcal{E}}_{s,t}^{\mathcal{X}}(X)$ as being the \textit{%
martingale} part of $\mathcal{E}_{s,t}^{\mathcal{X}}(X),$ in analogy to the
Doob-Meyer decomposition for semimartingales.

$\clubsuit \qquad (\mathbf{G)\quad }$We say that $\mathcal{X}_{s,t}$ has 
\textbf{linear growth }if there exists a constant $C_G $ such that for all $%
0 \le s \le t $ and for every $X\in L_{s}^{2}(H),$ 
\begin{eqnarray}
\| \widehat{\mathcal{E}}_{s,t}^{\mathcal{X}}(X) \|_2 &\leq &C_G (t-s) (1+\|
X\|_2 ) ,  \notag \\
\| \widetilde{\mathcal{E}}_{s,t}^{\mathcal{X}}(X) \|_2 &\le & C_G (
t-s)^{1/2} ( 1 + \|X\|_2) .  \label{sf0}
\end{eqnarray}

$\clubsuit \qquad (\mathbf{L})\qquad $We say that $\mathcal{X}_{s,t}$ is 
\textbf{Lipschitz continuous }if there exists a constant $C_{L} $ such that
for all $0 \le s \le t $ and for all $X,Y\in L_{s}^{2}(H),$ 
\begin{eqnarray}
i)\quad E\left\vert \widehat{\mathcal{E}}_{s,t}^{\mathcal{X}}(X)-\widehat{%
\mathcal{E}}_{s,t}^{\mathcal{X}}(Y)\right\vert ^{2} &\leq
&C_{L}(t-s)^{2}E\left\vert X-Y\right\vert ^{2} ,  \label{sf1} \\
ii)\quad E\left\vert \widetilde{\mathcal{E}}_{s,t}^{\mathcal{X}}(X)-%
\widetilde{\mathcal{E}}_{s,t}^{\mathcal{X}}(Y)\right\vert ^{2} &\leq
&C_{L}(t-s) E\left\vert X-Y\right\vert ^{2} .  \label{sf1'}
\end{eqnarray}

$\clubsuit \qquad \mathbf{(S(}\beta ))$ \textbf{Sewing property. }For $s<r<t$
we denote%
\begin{equation*}
\mathcal{E}_{s,r,t}^{\mathcal{X}}(X)=\mathcal{X}_{r,t}(\mathcal{X}%
_{s,r}(X))-X,
\end{equation*}%
and we apply the same operations $\widehat{\mathcal{E}}_{s,r,t}^{\mathcal{X}%
}(X)$ and $\widetilde{\mathcal{E}}_{s,r,t}^{\mathcal{X}}(X)$ to it as in (%
\ref{sf-1}). Let $\beta >0.$ We suppose that 
there exists a constant $C_{sew}$ 
such that 
\begin{eqnarray}
E\left\vert \widehat{\mathcal{E}}_{s,t}^{\mathcal{X}}(X)-\widehat{\mathcal{E}%
}_{s,r,t}^{\mathcal{X}}(X)\right\vert ^{2} &\leq &C_{sew}(1+E\left\vert
X\right\vert ^{2})(t-s)^{2+\beta },  \label{sl4} \\
E\left\vert \widetilde{\mathcal{E}}_{s,t}^{\mathcal{X}}(X)-\widetilde{%
\mathcal{E}}_{s,r,t}^{\mathcal{X}}(X)\right\vert ^{2} &\leq
&C_{sew}(1+E\left\vert X\right\vert ^{2})(t-s)^{1+\beta },  \label{sl4'}
\end{eqnarray}%
for every $0\leq s<r<t$ and every $X\in L_{s}^{2}(H).$

\begin{equation}  \label{Abeta}
\mbox{If $\mathbf{(G),(L),}$ and\textbf{\ }$\mathbf{(S(}\beta ))$ hold, we say
that $\mathbf{A(\beta })$ holds.}
\end{equation}

Let $s<t$ and let $\pi =\{s=t_{0}<t_{1}<...<t_{n}=t\}$ be a partition of the
interval $(s,t).$ We denote 
\begin{equation*}
\mathcal{X}_{s,t}^{\pi }=\mathcal{X}_{t_{n-1},t}\circ \mathcal{X}%
_{t_{n-2},t_{n-1}}\circ ....\circ \mathcal{X}_{s,t_{1}}.
\end{equation*}%
Given the partition $\pi,$ for all $0 \le i < n, $ we fix some $r_i $ with $%
t_{i}\leq r_{i}<t_{i+1}$ and define $\pi ^{\prime }=\{s=t_{0}\leq
r_{0}<t_{1}\leq r_{1}<t_{2}<...<t_{n}=t\}.$ We say that $\pi ^{\prime }$ is
a proper sub-partition of $\pi .$ The following lemma states a key estimate.

\begin{lemma}
Suppose that $\mathbf{A(\beta })$ holds for $\mathcal{X}$ and let $T>0.$ Let 
$0\leq s\leq t$ with $t-s\leq 1$ and $t\leq T,$ and consider a partition $%
\pi =\{s=t_{0}<t_{1}<...<t_{n}=t\}$ and a proper sub-partition $\pi ^{\prime
}$ of $\pi .$

\noindent \textbf{A }There exists a constant $C$ not depending on $\pi ,$
such that for any $X\in L_{s}^{2}(H),$ 
\begin{equation}
\left\Vert \mathcal{X}_{s,t}^{\pi }(X)\right\Vert _{2}^{2}\leq C\left(
(t-s)+\Vert X\Vert _{2}^{2}\right).  \label{sf15}
\end{equation}%
\textbf{B } There exists a constant $C_{T}$ not depending on $\pi $ nor on $%
\pi ^{\prime },$ such that for any $X^{1},X^{2}\in L_{s}^{2}(H),$ 
\begin{equation}
\left\Vert \mathcal{X}_{s,t}^{\pi }(X^{1})-\mathcal{X}_{s,t}^{\pi ^{\prime
}}(X^{2})\right\Vert _{2}^{2}\leq   C_{T}(\left\Vert X^{1}-X^{2}\right\Vert
_{2}^{2}+(t-s)\left( 1+\Vert X^{2}\Vert _{2}^{2}\right) \left\vert \pi
\right\vert ^{\beta }),  \label{s12}
\end{equation}%
with $\left\vert \pi \right\vert =\max_{i=1,...,n}(t_{i+1}-t_{i}).$
\end{lemma}

\begin{proof}
We take $X_{0}^{1}= X^1, X_0^2 = X^2 ,$ and we define by recurrence, for any 
$0 \le i < n,$ 
\begin{equation*}
X_{i+1}^{1}=\mathcal{X}_{t_{i},t_{i+1}}(X_{i}^{1}),\quad X_{i+1}^{2}=%
\mathcal{X}_{r_{i},t_{i+1}}(\mathcal{X}_{t_{i},r_{i}}(X_{i}^{2})).
\end{equation*}%
In particular, $\mathcal{X}_{s,t}^{\pi }(X^{1})=X_{n}^{1}$ and $\mathcal{X}%
_{s,t}^{\pi ^{\prime }}(X^{2})=X_{n}^{2}.$\ We first show that there exists
a constant $C $ such that for any $k \le n, $ 
\begin{equation}  \label{s12'}
\| X_k^1 \|_2^2 \le C (t-s) + C \| X^1 \|_2^2 , \; \| X_k^2 \|_2^2 \le C
(t-s) + C \| X^2 \|_2^2 .
\end{equation}
(\ref{sf15}) follows from this, taking $k=n.$

To prove the above inequality, we start from the decomposition 
\begin{eqnarray*}
X_k^1 &=&X^{1}+\sum_{i=0}^{k-1}(\mathcal{X}%
_{t_{i},t_{i+1}}(X_{i}^{1})-X_{i}^{1})=X^{1}+\sum_{i=0}^{k-1}\mathcal{E}^{%
\mathcal{X}}_{t_{i},t_{i+1}}(X_{i}^{1}) \\
&=& X^{1}+\sum_{i=0}^{k-1}\widehat{ \mathcal{E}^{\mathcal{X}}_{t_{i},t_{i+1}}%
}(X_{i}^{1}) + \sum_{i=0}^{k-1}\widetilde {\mathcal{E}^{\mathcal{X}%
}_{t_{i},t_{i+1}}}(X_{i}^{1}) .
\end{eqnarray*}%
By orthogonality and using the growth condition, we have that 
\[
\| \sum_{i=0}^{k-1}\widetilde {\mathcal{E}^{\mathcal{X}} _{t_{i},t_{i+1}}}%
(X_{i}^{1}) \|_2^2 = \sum_{i=0}^{k-1} \| \widetilde {\mathcal{E}^{\mathcal{%
X}}_{t_{i},t_{i+1}}}(X_{i}^{1}) \|_2^2 
\le  C \sum_{i=1}^{k-1} (t_{i+1} - t_i) ( 1 + \| X_i^1\|_2^2 ) ,
\]
where the constant $C$ does only depend on $C_G . $ Let us now consider the 
\textit{bounded-variation} part. Using Jensen's inequality and the growth
condition, we have 
\begin{eqnarray*}
\| \sum_{i=0}^{k-1}\widehat{ \mathcal{E}^{\mathcal{X}} _{t_{i},t_{i+1}}}%
(X_{i}^{1}) \|_2^2 &=& \| (t_k - t_0) \sum_{i=0}^{k-1} \frac{ t_{i+1} - t_i}{%
t_k - t_0} \frac{\widehat{ \mathcal{E}^{\mathcal{X}} _{t_{i},t_{i+1}}}%
(X_{i}^{1}) }{t_{i+1} - t_i } \|_2^2 \\
&\le& (t_k - t_0 ) \sum_{i=0}^{k-1} (t_{i+1} - t_i ) \| \frac{\widehat{ 
\mathcal{E}^{\mathcal{X}} _{t_{i},t_{i+1}}}(X_{i}^{1}) }{t_{i+1} - t_i }
\|_2^2 \\
& \le & (t_k - t_0 ) \sum_{i=0}^{k-1} (t_{i+1} - t_i ) C ( 1 + \| X_i^1
\|_2^2 ),
\end{eqnarray*}
with $C$ depending only on $C_G.$ Putting things together, we see that for
any $k \le n,$ 
\begin{equation*}
\| X_k^1\|_2^2 \le 3 \|X^1\|_2^2 + 3 C ( t_k -t_0) + 6 C \sum_{i=0}^{k-1}
(t_{i+1} - t_i ) \| X_i^1 \|_2^2 .
\end{equation*}
Using a discrete time version of Gronwall's lemma, we deduce that 
\begin{equation*}
\| X_k^1\|_2^2 \le \left( 3 \|X\|_2^2 + 3 C (t-s) \right) e^{ 6C (t-s) } ,
\end{equation*}
which, by increasing the constant if necessary, implies (\ref{s12'}) for $%
\|X_k^1 \|_2^2 $ since $(t-s) \le 1. $ The assertion for $\| X_k^2 \|_2^2 $
follows analogously.

We now prove (\ref{s12}) and start investigating the scheme associated to $%
\pi^{\prime}.$ In the partition $\pi ^{\prime }, $ to go from $i $ to $i+1, $
we perform two steps of the underlying discrete scheme, and so we have for
any $k \le n , $ 
\begin{equation*}
X_k^2 =X^{2}+\sum_{i=0}^{k-1}(\mathcal{X}_{r_{i},t_{i+1}}(\mathcal{X}%
_{t_{i},r_{i}}(X_{i}^{2}))-X_{i}^{2})=X^{2}+\sum_{i=0}^{k-1}\mathcal{E}^{%
\mathcal{X}}_{t_{i},r_{i},t_{i+1}}(X_{i}^{2}).
\end{equation*}%
We denote%
\begin{equation*}
\delta _{i}=\mathcal{E}^{\mathcal{X}}_{t_{i},t_{i+1}}(X_{i}^{1})-\mathcal{E}%
^{\mathcal{X}}_{t_{i},r_{i},t_{i+1}}(X_{i}^{2}),
\end{equation*}%
so that for every $k \le n ,$ 
\begin{equation*}
X_k^1 - X_k^2 =X^{1}-X^{2}+\sum_{i=0}^{k-1}\delta _{i}.
\end{equation*}%
Using the Lipschitz property and the sewing property together with the
estimates (\ref{s12'}), we get%
\begin{eqnarray*}
\left\Vert E_{t_i}(\delta _{i})\right\Vert _{2} &\leq &C (t_{i+1}-t_{i})
\left\Vert X_{i}^{1}-X_{i}^{2}\right\Vert _{2} + 
 C \sqrt{ (t-s) + \|X^2 \|_2^2 }
\; ( t_{i+1}-t_{i})^{1+ \beta /2}, \\
\left\Vert \widetilde{E}_{t_i}(\delta _{i})\right\Vert _{2} &\leq &C
(t_{i+1}-t_{i})^{1/2} \left\Vert X_{i}^{1}-X_{i}^{2}\right\Vert _{2} 
 + C \sqrt{ (t-s) + \|X^2 \|_2^2 
} \; ( t_{i+1}-t_{i} )^{ (1+\beta) /2}.
\end{eqnarray*}%
Using the same trick as above to obtain the estimate (\ref{s12'}), it
follows that 
\begin{eqnarray*}
\left\Vert \sum_{i=0}^{k-1}E_{t_i}(\delta _{i})\right\Vert _{2}^{2} &\leq &
(t-s) \sum_{i=0}^{k-1}(t_{i+1}-t_{i})\left\Vert
(t_{i+1}-t_{i})^{-1}E_{t_i}(\delta _{i})\right\Vert _{2}^{2} \\
&\leq & C ( t-s) \sum_{i=0}^{k-1} (t_{i+1}-t_{i}) \Big \{ \left\Vert
X_{i}^{1}-X_{i}^{2}\right\Vert _{2}^{2} \\
&& \quad \quad \quad \quad \quad \quad \quad \quad\quad + \left[ t-s + \|X^2
\|_2^2 \right] (t_{i+1}-t_{i})^{ \beta } \Big\} \\
&\leq & C \sum_{i=0}^{k-1} (t_{i+1}-t_{i}) \Big \{ \left\Vert
X_{i}^{1}-X_{i}^{2}\right\Vert _{2}^{2} \\
&& \quad \quad \quad \quad \quad \quad \quad \quad\quad \quad \quad \quad + 
\left[ 1 + \|X^2 \|_2^2 \right] (t_{i+1}-t_{i})^{ \beta } \Big\} ,
\end{eqnarray*}%
where the last inequality follows from $t-s \le 1.$ Moreover, using
orthogonality, 
\begin{eqnarray*}
\left\Vert \sum_{i=0}^{k-1}\widetilde{E}_{t_i}(\delta _{i})\right\Vert
_{2}^{2}&=&\sum_{i=0}^{k-1}\left\Vert \widetilde{E}_{t_i}(\delta
_{i})\right\Vert _{2}^{2} \\
& \leq & C \sum_{i=0}^{k-1} (t_{i+1}-t_{i}) \Big \{ \left\Vert
X_{i}^{1}-X_{i}^{2}\right\Vert _{2}^{2} + \left[ 1 + \|X^2 \|_2^2 \right]
(t_{i+1}-t_{i})^{ \beta } \Big\} .
\end{eqnarray*}%
Since $t_{i+1} -t_i \le | \pi | , $ we conclude that for any $k \le n , $ 
\begin{eqnarray*}
\left\Vert X_{k}^{1}-X_{k}^{2}\right\Vert _{2}^{2} &\le & \left\Vert
X^{1}-X^{2}\right\Vert _{2}^{2}+ C (t-s ) \left[ 1 + \|X^2 \|_2^2 \right] |
\pi | ^{\beta } \\
&&\quad \quad \quad \quad \quad\quad \quad \quad + C
\sum_{i=0}^{k-1}(t_{i+1}-t_{i}) \left\Vert X_{i}^{1}-X_{i}^{2}\right\Vert
_{2}^{2} .
\end{eqnarray*}%
Using once more a discrete time version of Gronwall's lemma, it follows that
for any $k \le n , $ 
\begin{equation*}
\left\Vert X_{k}^{1}-X_{k}^{2}\right\Vert _{2}^{2} \le e^{ C (t_k - t_0) }
\left( \left\Vert X^{1}-X^{2}\right\Vert _{2}^{2}+ C (t-s ) \left[ 1 + \|X^2
\|_2^2 \right] | \pi | ^{\beta } \right) ,
\end{equation*}
and evaluating this inequality for $k = n $ yields 
\begin{equation*}
\left\Vert \mathcal{X}_{s,t}^{\pi }(X^{1})-\mathcal{X}_{s,t}^{\pi ^{\prime
}}(X^{2})\right\Vert _{2}^{2}\leq C_{T}(\left\Vert X^{1}-X^{2}\right\Vert
_{2}^{2}+ (t-s ) \left[ 1 + \|X^2 \|_2^2 \right] \left\vert \pi \right\vert
^{\beta }),
\end{equation*}%
whence the assertion.
\end{proof}

\bigskip In the sequel, we will use particular dyadic partitions which are
defined as follows: Given a fixed $n\in {\mathbb{N}},$ we denote $t_{k}=%
\frac{k}{2^{n}},$ and given $s\geq 0$ we denote $k(s)$ the unique index such
that $t_{k(s)}\leq s<t_{k(s)+1}.$ Then for fixed $0\leq s<t$ with $t-s\leq
1, $ we consider the partition 
\begin{equation}  \label{eq:pin}
\pi _{n}(s,t)=\{s<t_{k(s)+1}<...<t_{k(t)}\le t\}
\end{equation}
of the interval $(s,t).$ We denote 
\begin{equation}  \label{eq:Xn}
\mathcal{X}_{s,t}^{n}=\mathcal{X}_{s,t}^{\pi _{n}(s,t)}=\mathcal{X}%
_{t_{k(t)},t}\circ \mathcal{X}_{t_{k(t)-1}, t_{k(t)}}\circ ....\circ 
\mathcal{X}_{s,t_{k(s)+1}}.
\end{equation}

We will now state a version of the stochastic sewing lemma given by L\^e,
which is appropriate in our specific framework. Other variants are given and
used in the literature, see e.g. \cite{Le} and \cite{ABC} .

To state the sewing lemma, we have to introduce the notion of \textbf{%
stochastic flow}.

\begin{definition}
\textbf{(stochastic flow) } Let $\mathbb{X}_{s,t}$ be an application from $%
L_{s}^{2}(H)$ to $L_{t}^{2}(H),$ for any $0\leq s<t < \infty.$ We say that $%
\mathbb{X}_{s,t}$ is a \textbf{stochastic flow} if for every $0 \le s<r<t,$
and $X\in L_{s}^{2}(H)$ we have $\mathbb{X}_{s,t}(X)=\mathbb{X}_{r,t}(%
\mathbb{X}_{s,r}(X)).$
\end{definition}

\begin{lemma}
\label{sewing}\textbf{A }Suppose that $\mathcal{X}$ verifies $\mathbf{%
A(\beta }).$ Then there exists a unique stochastic flow $\mathbb{X}$ and a
constant $C(\beta ),$ depending on the constants in $\mathbf{A(\beta }),$\
such that for all $0\leq s\leq t,$ $t-s\leq 1,$ 
\begin{equation}
E\left\vert \mathbb{X}_{s,t}(X)-\mathcal{X}_{s,t}^{n}(X)\right\vert ^{2}\leq
C(\beta )(1+E\left\vert X\right\vert ^{2})(t-s)2^{-n\beta }.  \label{sl3}
\end{equation}%
\textbf{B }Moreover, $\mathbb{X}$ verifies $\mathbf{(G)}$ and $\mathbf{(L)}.$
\newline
\textbf{C }If $\mathcal{X}$ verifies $\mathbf{A(\beta })$ and if $\mathbb{X}$
is the associated stochastic flow, then we have more generally that for all $%
0\leq s\leq t,$ $t-s\leq 1,$ and for any partition $\pi $ of $(s,t),$ 
\begin{equation}
E\left\vert \mathbb{X}_{s,t}(X)-\mathcal{X}_{s,t}^{\pi }(X)\right\vert
^{2}\leq C(\beta )(1+E\left\vert X\right\vert ^{2})(t-s)\left\vert \pi
\right\vert ^{\beta }.  \label{s14}
\end{equation}
\end{lemma}

\begin{definition}
\bigskip We call $\mathbb{X}$ the ``flow solution" associated to $\mathcal{X}%
. $
\end{definition}

\begin{proof}
We fix $s<t$ and we will prove that the sequence of Euler schemes $\mathcal{X%
}_{s,t}^{n}(X),n\in {\mathbb{N}},$ is Cauchy in $L_{t}^{2},$ and then we
define $\mathbb{X}_{s,t}(X)$ as the limit. Notice that $\pi _{n+1}(s,t)$ is
a proper sub-partition of $\pi _{n}(s,t).$ Using the previous lemma and
applying (\ref{s12}) with $X^{1}=X^{2}=X$ gives 
\begin{equation*}
\left\Vert \mathcal{X}_{s,t}^{n}(X)-\mathcal{X}_{s,t}^{n+1}(X)\right\Vert
_{2}^{2}\leq C(t-s)(1+\Vert X\Vert _{2}^{2})2^{-\beta n}.
\end{equation*}%
Fix some $n_{0}.$ The above inequality holds for every $n\geq n_{0},$ such
that for every $n_{0}\leq n\leq m ,$ 
\begin{eqnarray}
\left\Vert \mathcal{X}_{s,t}^{n}(X)-\mathcal{X}_{s,t}^{m}(X)\right\Vert _{2}
&\leq &\sum_{k=n}^{m}\left\Vert \mathcal{X}_{s,t}^{k}(X)-\mathcal{X}%
_{s,t}^{k+1}(X)\right\Vert _{2}  \notag  \label{sl1} \\
&\leq &C\sqrt{t-s}\sqrt{1+\Vert X\Vert _{2}^{2}}\;\sum_{k=n}^{m}2^{-(\beta
/2)k}  \notag \\
&\leq &C\sqrt{t-s}\sqrt{1+\Vert X\Vert _{2}^{2}}\;2^{-(\beta /2)n}
\end{eqnarray}%
(recall that the value of a constant is allowed to change from occurrence to
occurrence). So we have a Cauchy sequence in $L_{t}^{2},$ and we may define $%
\mathbb{X}_{s,t}(X): =\lim_{n}\mathcal{X}_{s,t}^{n}(X).$

Fix now some $s<r<t.$ By considering the dyadic partitions $\pi_n ( s, r)$
and $\pi_n ( r, t), $ we may define 
\begin{equation*}
\delta _{s,r,t}(\mathcal{X}^{n}(X))=\mathcal{X}_{s,t}^{n}(X)-\mathcal{X}%
_{r,t}^{n}(\mathcal{X}_{s,r}^{n}(X)).
\end{equation*}%
Since the concatenation of $\pi_n ( s, r)$ and $\pi_n ( r, t)$ is a proper
sub-partition of $\pi_n ( s,t), $ using (\ref{s12}) and passing to the
limit, we obtain 
\begin{equation*}
E(\left\vert \delta _{s,r,t}(\mathbb{X}(X))\right\vert
^{2})=\lim_{n}E(\left\vert \delta _{s,r,t}(\mathcal{X}^{n}(X))\right\vert
^{2})=0.
\end{equation*}%
As a consequence, $\mathbb{X}_{s,t}$ has the flow property. 

Following the arguments in \cite{ABC}, it is straightforward to show that
for all $s<t,$ $\mathbb{X}_{s,t} $ also satisfies (\ref{s12'}) as well as
the Lipschitz property 
\begin{equation*}
\left\Vert \mathbb{X}_{s,t}(X)-\mathbb{X}_{s,t}(Y)\right\Vert _{2}^{2}\leq
C_{T}^{\prime }\left\Vert X-Y\right\Vert _{2}^{2}.
\end{equation*}

Finally, the proof of \textbf{C } is a straightforward adaptation of the
arguments of the proof of Theorem 2.10 in \cite{ABC}.
\end{proof}

\section{\protect\bigskip Particle systems}

\label{sec:4}

\subsection{Framework}

\label{sec:frameworkgeneral} To represent the processes we are interested in
we fix a generic probability space $(\Omega ,\mathcal{F},P)$ equipped with a
filtration $(\mathcal{F}_{t})_{t\geq 0}.$ We also fix some $d_{0}\leq d$ and
we assume that on $\Omega $ is defined, for any $s \geq 0, $ a $d_{0}-$%
dimensional process $W_{s,t}=(W_{s,t}^{1},...,W_{s,t}^{d_{0}}), t \geq s, $
such that $W_{s,t } $ is centered and $\mathcal{F}_t-$measurable for any $t
\geq s.$ In the main example discussed in the present paper, we will choose $%
W_{s,t} = W_t- W_s,$ where $(W_t)_{t \geq 0} $ is a standard Brownian motion. 

We denote by $\mathcal{F}_{t}^{W}=\sigma (W_{s,r}, 0 \le s\leq r\leq t),$
for any $t\geq 0.$ $W$ will be a common noise underlying each particle's
evolution within an associated limit particle system. We recall that $%
1_{d_{0}}w\in {\mathbb{R}}^{d}$ denotes the vector $1_{d_{0}}w=(0,\ldots
,0,w^{1},...,w^{d_{0}}).$

We also consider a family of filtrations $(\mathcal{G}^1_{s,t})_{t\geq s\geq
0}, (\mathcal{G}^2_{s,t})_{t \geq s \geq 0}$ such that for any fixed $s, $ $(%
\mathcal{G}^i_{s,t})_{ t \geq s}, i=1, 2, $ is a filtration, and such that
moreover the following properties are satisfied.

\begin{enumerate}
\item $\mathcal{G}^i_{s,t}, i= 1, 2,$ are independent.

\item $\mathcal{G}^i_{s,t}, i= 1, 2,$ are both independent of $\mathcal{F}%
_{s}$ and of $\mathcal{F}_{t}^{W}.$

\item $\mathcal{G}^i_{s,t}, i= 1, 2,$ are both included in $\mathcal{F}_{t}.$
\end{enumerate}

We think of $\mathcal{G}_{s,t}^{1}$ and $\mathcal{G}_{s,t}^{2}$ as
representing two sources of \textit{innovation} within the period $(s,t),$
which are independent of the past $\mathcal{F}_{s}$ and independent of the
process $W.$ The first source of innovation, which is $\mathcal{G}_{s,t}^{1}-
$adapted, will be related to an autonomous evolution of the different
particles, see (\ref{eq:H}) below, while the second source of innovation,
which is $\mathcal{G}_{s,t}^{2}-$adapted, will be a perturbation term to
which each particle is equally exposed, see (\ref{pf4}) below.

Recall that $L_{t}^{2}({\mathbb{R}}^{d\times N})$ denotes the space of $%
\mathcal{F}_{t}-$\ measurable, square integrable ${\mathbb{R}}^{d\times N}-$%
valued random variables.

In what follows, $d$ and $d_0$ will be fixed and $N , $ the number of
particles, will tend to infinity. Therefore, we consider on ${\mathbb{R}}%
^{d\times N}$ the following norm: Writing any element $x \in {\mathbb{R}}^{
d \times N} $ as $x =(x^{1},...,x^{N})$ with $x^{i}\in {\mathbb{R}}^{d},$ we
define 
\begin{equation*}
\left\vert x\right\vert _{N}^{2}=\frac{1}{N}\sum_{i=1}^{N}\left\vert
x^{i}\right\vert ^{2}.
\end{equation*}%
We denote by $\mathcal{P}_{2}({\mathbb{R}}^{d})$ ($\mathcal{P}_{2}({\mathbb{R%
}}^{d_0}),$ respectively) the space of all probability measures on ${\mathbb{%
R}}^{d}$ (on ${\mathbb{R}}^{d_0}, $ respectively) having a finite second
moment, and by $\mathcal{P}_{2}({\mathbb{R}}^{d\times N})$ the space of all
probability measures $\mu $ on ${\mathbb{R}}^{d\times N}$ such that $\int_{{%
\mathbb{R}}^{d\times N}}\left\vert x\right\vert _{N}^{2}\mu (dx)<\infty .$
On $\mathcal{P}_{2}({\mathbb{R}}^{d\times N}),$ we consider the associated $%
\mathcal{W}_{2}-$ Wasserstein distance given by 
\begin{equation*}
\mathcal{W}_{2}(\mu ,\nu )=\inf_{\Pi }\left \{\int_{{\mathbb{R}}^{d\times
N}}\int_{{\mathbb{R}}^{d\times N}}\left\vert x-y\right\vert _{N}^{2}\Pi
(dx,dy)\right\} ,
\end{equation*}%
where the infimum is taken over all probability measures on ${\mathbb{R}}%
^{d\times N}\times {\mathbb{R}}^{d\times N}$ having marginals $\mu $ and $%
\nu .$ We keep the notation $W_{2}$ (using the non-calligraphic letter $W_2$%
) for the classical Wasserstein distance on $\mathcal{P}_{2}({\mathbb{R}}%
^{d})$ given by 
\begin{equation*}
W_{2}(\mu ,\nu )=\inf_{\Pi }\left\{\int_{{\mathbb{R}}^{d}}\int_{{\mathbb{R}}%
^{d}}\left\vert x-y\right\vert _{N}^{2}\Pi (dx,dy) \right\} ,
\end{equation*}%
where $\Pi $ has marginals $\mu $ and $\nu .$

Finally, for two ${\mathbb{R}}^{d}-$valued square integrable random
variables $X$ and $Y,$ by abuse of notation, we denote 
\begin{equation*}
W_{2}(X,Y)=W_{2}(\mathcal{L}(X),\mathcal{L}(Y)),
\end{equation*}%
and similarly, for two ${\mathbb{R}}^{d\times N}-$ valued square integrable
random variables $X$ and $Y,$ 
\begin{equation*}
\mathcal{W}_{2}(X,Y)=\mathcal{W}_{2}(\mathcal{L}(X),\mathcal{L}(Y)).
\end{equation*}%
For $X\in L^{2}({\mathbb{R}}^{d\times N}),$ we denote 
\begin{equation*}
E_{s}(X)=E(X\mid \mathcal{F}_{s})\quad \mbox{and}\quad \quad \widetilde{E}%
_{s}(X)=X-E_{s}(X).
\end{equation*}%
Finally we define%
\begin{equation*}
\mu_{N}(x,dz):=\frac{1}{N}\sum_{i=1}^{N}\delta _{x^{i}}(dz), \,
x=(x^{1},...,x^{N}) \in {\mathbb{R}}^{ d \times N}, z \in {\mathbb{R}}^d,
\end{equation*}%
and for a function $\varphi :{\mathbb{R}}^{d}\rightarrow {\mathbb{R}}^{m},$
we denote%
\begin{equation*}
\mu _{N}(x,\varphi )=\int \varphi (z)\mu_{N}(x,dz)=\frac{1}{N}%
\sum_{i=1}^{N}\varphi (x^{i}).
\end{equation*}

We introduce now the core objects of our paper.

$\clubsuit $ \qquad \textbf{Autonomous innovation.} For every $0 \le s<t, $
we consider $H_{s,t}^{i}:\Omega \times {\mathbb{R}}^{d}\times {\mathbb{R}}%
^{m}\rightarrow {\mathbb{R}}^{d},i=1,...,N,$ which are $\mathcal{G}^1_{s,t}
\otimes \mathcal{B} ( {\mathbb{R}}^d \times {\mathbb{R}}^m) - \mathcal{B} ( {%
\mathbb{R}}^d )-$measurable and satisfy 
\begin{equation}  \label{eq:H}
H_{s,t}^{i}( X,a)\in L_{t}^{2}({\mathbb{R}}^{d}) \mbox{ whenever  } ~X\in
L_{s}^{2}({\mathbb{R}}^{d}) \mbox{ and  } a\in L_{s}^{2}({\mathbb{R}}^{m}) .
\end{equation}%
We put $H_{s,t}=(H_{s,t}^{1},...,H_{s,t}^{N})\in {\mathbb{R}}^{d\times N}.$
Since $\mathcal{G}^1_{s,t} $ is independent of $\mathcal{F}_t^W, $ for any
fixed $x, a, $ $H_{s, t }^i (x , a ) $ is independent of $W.$

We impose the following assumptions on $ H_{s, t}$ which are natural and 
straightforward to check in most examples. 

\begin{assumption}
\label{ass:H}

\begin{enumerate}
\item If $X^{1},...,X^{N} \in L_{s}^{2}({\mathbb{R}}^{d})$ are i.i.d. and if 
$a\in {\mathbb{R}}^{m}$ is deterministic, then $H_{s,t}^{i}(X^{i},a),$ $%
i=1,...,N,$ are i.i.d.

\item We assume that there exists a constant $C_{stb}$\ such that, almost
surely, for every $x,y\in {\mathbb{R}}^{d},a,b\in {\mathbb{R}}^{m},$ 
\begin{equation}
\left\vert E_s\left(H_{s,t}^{i}(x,a)-H_{s,t}^{i}(y,b)\right)\right\vert \leq
C_{stb}(t-s)(\left\vert x-y\right\vert +\left\vert a-b\right\vert ).
\label{pf1}
\end{equation}%
Moreover, for every $X^i,Y^i\in L_{s}^{2}({\mathbb{R}}^{d}),A,B\in L_{s}^{2}(%
{\mathbb{R}}^{m}),$%
\begin{equation}
E\left\vert \widetilde{E}_{s}(H_{s,t}^{i}(X^i,A)-H_{s,t}^{i}(Y^i,B))\right%
\vert ^{2}\leq C^2_{stb}(t-s)(E\left\vert X^i-Y^i\right\vert
^{2}+E\left\vert A-B\right\vert ^{2}) .  \label{pf1'}
\end{equation}
\end{enumerate}
\end{assumption}

\begin{example}
Let $N^i, 1 \le i \le N, $ be independent Poisson random measures on ${%
\mathbb{R}}_+ \times {\mathbb{R}}^d $ having intensity $dt \nu ( dv) $ where 
$\nu $ is a $\sigma-$finite measure having a finite second moment. Let
moreover $B^i, 1 \le i \le N, $ be independent $d-$dimensional Brownian
motions, independent of the Poisson random measures. Let $\psi : {\mathbb{R}}%
^d \times {\mathbb{R}}^m \to {\mathbb{R}}^d $ and $\sigma : {\mathbb{R}}^d
\to {\mathbb{R}}^{d \times d } $ be measurable and consider 
\begin{equation*}
H_{s,t}^{i}(x,a)=\psi (x,a)(t-s)+\int_{] s, t ] } \int_{{\mathbb{R}}^d
}vN^{i}(dr,dv)+\sigma (x)(B_{t}^{i}-B_{s}^{i}).
\end{equation*}
\end{example}

Suppose now that  $X = ( X^1, \ldots, X^N),$ with $X^i \in L^2_s ( {\mathbb{R}}^d) $
for any $i.$ We \ fix a bounded Lipschitz continuous function $\varphi :{%
\mathbb{R}}^{d}\rightarrow {\mathbb{R}}^{m}$ and we define $h_{s,t}^{i}
:L_{s}^{2}({\mathbb{R}}^{d\times N})\rightarrow {\mathbb{R}}^{d},i=1,...,N,$
by 
\begin{equation}
h_{s,t}^{i}(X) =H_{s,t}^{i}(X^{i},\mu_{N}(X, \varphi ))=H_{s,t}^{i}(X^{i},%
\frac{1}{N}\sum_{j=1}^{N}\varphi (X^{j})) .  \label{pf2}
\end{equation}

In the sequel, we will mostly consider the situation when $(X^{1},\ldots
,X^{N})$are conditionally independent, conditionally on $\mathcal{F}_{s}^{W}.
$ In this case, we also introduce $\widehat{h}_{s,t}^{i}:L_{s}^{2}({\mathbb{R%
}}^{d\times N})\rightarrow {\mathbb{R}}^{d},i=1,...,N,$ by 
\begin{equation}
\widehat{h}_{s,t}^{i}(X)=H_{s,t}^{i}(X^{i},E(\varphi (X^{i})|\mathcal{F}%
_{s}^{W})).  \label{pf3}
\end{equation}%
We denote $h(X)=(h^{1}(X),...,h^{N}(X))\in {\mathbb{R}}^{d\times N},$ and $%
\widehat{h}(X)=(\widehat{h}^{1}(X),...,\widehat{h}^{N}(X)).$ Notice that if $%
X = (X^1, \ldots, X^N) $ is exchangeable, then both $h(X) $ and $\widehat{h}
( X )$ are exchangeable, too.

$\clubsuit $ \qquad \textbf{Perturbations. } For every $0 \le s < t, $ we
consider random fields 
\begin{equation}  \label{pf4}
S_{s,t}^{N}(x) = S_{s,t}^{N}(x,\omega ) \mbox{ and  } \overline{S}_{s, t }^N
( x ) = \overline{S}_{s,t}^{N}(x,\omega ) \in {\mathbb{R}}^{d_{0}},\;
x=(x^{1},...,x^{N})\in {\mathbb{R}}^{d\times N}, \omega \in \Omega ,
\end{equation}
satisfying the following assumption.

\begin{assumption}
\label{ass:perturbation} i) For any $0 \le s < t, $ we have that 
\begin{equation*}
\Omega \times {\mathbb{R}}^{d \times N }\ni (\omega ,x )\mapsto S_{s, t}^N
(\omega ,x)\in {\mathbb{R}}^{d_{0}}
\end{equation*}%
is $\mathcal{G}_{s,t}^{2}\otimes \mathcal{B}({\mathbb{R}}^{d \times N })-%
\mathcal{B}({\mathbb{R}}^{d_0})-$ measurable and 
\begin{equation*}
\Omega \times {\mathbb{R}}^{d \times N }\ni (\omega ,x )\mapsto \overline
S_{s, t}^N (\omega ,x)\in {\mathbb{R}}^{d_{0}}
\end{equation*}%
is $\mathcal{F}^W_{t}\otimes \mathcal{B}({\mathbb{R}}^{d \times N })-%
\mathcal{B}({\mathbb{R}}^{d_0})-$ measurable. Moreover, for any fixed $x,$ $%
\overline S_{s, t}^N (\omega ,x) $ is independent of $\mathcal{F}_s^W.$ 

ii) For any fixed $x\in {\mathbb{R}}^{d\times N},$ the random variables $%
S_{s, t}^N ( x, \omega) $ and $\overline{S}_{s,t}^{N}(x,\omega ) $ 
are centered, square integrable and independent of $\mathcal{F}_{s} .$

iii) For some finite constant $M_{2}, $ we have that for all $0 \le s < t $
and for all $x \in {\mathbb{R}}^{d \times N}, $ 
\begin{equation}
E\left\vert S^N_{s,t}(x,\omega )\right\vert ^{2} + E\left\vert \overline
S^N_{s,t}(x, \omega )\right\vert ^{2} \leq M_{2}| t-s| .  \label{MD0}
\end{equation}
iv) \textbf{$L^2-$continuity and Lipschitz property.} We have that 
\begin{equation}
E\left\vert \overline{S}_{s,t}^{N}(x,\omega )-\overline{S}_{s,t}^{N}(%
\overline{x},\omega )\right\vert ^{2}\leq \overline{L}(t-s) |x - \bar x
|_N^2.  \label{Aiii}
\end{equation}
v) \textbf{Coupling.} For all $0 \le s \le t, $ and for all $x \in {\mathbb{%
R}}^{d \times N}, $ we have that 
\begin{equation}
W_{2}^{2}(S_{s,t}^{N}(x,\omega ),\overline{S}_{s,t}^{N}(x,\omega ))\leq (t-s)%
\overline{\varepsilon }_{N} ,  \label{Aiv}
\end{equation}%
where $\lim_{ N \to \infty } \bar \varepsilon_N = 0.$
\end{assumption}
Items $i)$ to $iv)$ of this assumption are natural and easy to check in most examples.
The fact that the moments of order two scale as a factor of $t-s$  means that $S_{s,t}^{N}$
and $\overline{S}_{s, t }^N$ have to be thought of as being increments of square integrable martingales.
In our main example, the passage from $S_{s,t}^{N}(x,\omega )$ to $ \overline{S}_{s,t}^{N}(x,\omega )$ is
a consequence of the CLT, and in this case we have $ \bar \varepsilon_N = N^{ - 1 }, $ see Section \ref{sec:CLT} below. 
In a general framework, condition (\ref{Aiv}) has to be checked in a case by case
study, depending on the precise model assumptions. 

Here is an example illustrating the CLT framework.

\begin{example}
\label{exa:CLT} Let $B^{1},\ldots ,B^{N}$ are independent $d_{0}-$%
dimensional Brownian motions and put 
\begin{equation*}
S_{s,t}^{N}(x)=\frac{1}{\sqrt{N}}\sum_{i=1}^{N} \sqrt{f(x^{i})}%
(B_{t}^{i}-B_{s}^{i}),
\end{equation*}%
where $f$ is a positive definite $d_{0}\times d_{0}$ matrix and where $\sqrt{%
f}$ denotes its matrix square root.
\end{example}

In Example \ref{exa:CLT}, the limit common noise is given by Brownian motion, and this framework will be discussed in detail in Section 
\ref{sec:CLT} below. 

In what follows we state our main results in the general setting of the present section to provide a flexible framework capable to cover 
other situations and in order to streamline the main arguments.

The following assumption will allow us to perform the transition $ \bar X^N \to \hat X^N$ mentioned in the introduction, that is, to
compare the particle system with an infinite exchangeable limit system.

\begin{assumption}
\label{ass:perturbationbis} i) \textbf{Representation.} There exists some
measurable function $A_N : {\mathbb{R}}^{ d \times N} \to {\mathbb{R}}^{ d_0
\times d_0} $ such that 
\begin{equation}
\overline{S}_{s,t}^{N}(x,\omega )= A_N ( x) W_{s,t} .  \label{Avi}
\end{equation}
ii) \textbf{Conditional law of large numbers.} There exists a continuous and
bounded function $f : {\mathbb{R}}^{d} \to {\mathbb{R}}^{d_0 \times d_0},$
taking values in the set of symmetric positive definite $d_0 \times d_0-$%
matrices, and some exponent $\kappa > 0 , $ such that the
following holds: If $X = ( X^{1},\ldots ,X^{N}) \in L^2_s( {\mathbb{R}}^{ d
\times N} ) $ are conditionally independent, conditionally on $\mathcal{F}%
_{s}^{W},$ then we have almost sure convergence, as $ N \to \infty, $ 
\begin{equation*}
 A_N ( X) \to ({E} (f ( X^1) | \mathcal{F}_s^W ))^\kappa.
\end{equation*}
Moreover, putting 
\begin{equation}
\widehat{S}_{s,t}^{N}(X,\omega )= (E (f ( X^1) | \mathcal{F}_s^W
))^\kappa W_{s, t},  \label{Avibis}
\end{equation}
we have that 
\begin{equation}
E\left\vert \widehat{S}_{s,t}^{N}(X,\omega )-\overline{S}_{s,t}^{N}(X,\omega
)\right\vert ^{2}\leq (t-s)\widehat{\varepsilon }_{N} ,  \label{Av}
\end{equation}%
where $\lim_{ N \to \infty } \widehat \varepsilon_N = 0.$\newline
\end{assumption}

\begin{remark}
Notice that $\widehat{S}_{s,t}^{N}(X,\omega ) $ is not a function of $X (
\omega); $ it is only a function of its conditional law $\mathcal{L }( X^1 | \mathcal{F}_s^W).$
\end{remark}

\begin{example}
In the framework of Example \ref{exa:CLT}, we have that  
\begin{equation*}
\overline{S}_{s, t }^N ( x ) = \sqrt{ \mu_N ( x, f) } (W_t- W_s),
\end{equation*}
where $\sqrt{ \cdot } $ denotes the matrix square root of 
symmetric positive definite $d_0 \times d-0-$matrix. In particular, the above assumption holds with $\kappa = 1/2,$ and 
$$ \widehat{S}_{s,t}^{N}(X,\omega )= \sqrt{E (f ( X^1) | \mathcal{F}_s^W)} (W_t- W_s).$$
\end{example}

\subsection{Associated one step Euler schemes and main result}

We are now able to introduce the one step Euler schemes for the particle
systems that we will use in the sequel. For any exchangeable $%
X=(X^{1},...,X^{N})\in {\mathbb{R}}^{d\times N}$ with $X^{i}\in L_{s}^{2}({%
\mathbb{R}}^{d})$ for any $1\leq i\leq N,$ we introduce 
\begin{eqnarray}
\mathcal{X}_{s,t}^{N,i}(X)
&=&X^{i}+h_{s,t}^{i}(X)+1_{d_{0}}S_{s,t}^{N}(X),\quad i=1,...,N,  \label{pf7}
\\
\overline{\mathcal{X}}_{s,t}^{N,i}(X) &=&X^{i}+h_{s,t}^{i}(X)+1_{d_{0}}%
\overline{S}_{s,t}^{N}(X),\quad i=1,...,N,  \label{pf8} \\
\widehat{\mathcal{X}}_{s,t}^{N,i}(X) &=&X^{i}+\widehat{h}%
_{s,t}^{i}(X^{i})+1_{d_{0}}\widehat{S}_{s,t}(X) ,\quad i=1,...,N.
\label{pf9}
\end{eqnarray}%
The fact that we use $1_{d_{0}}$ implies that the input given by the
perturbations acts only on the last $d_{0}$ components and is null on the
first $d-d_{0}$ components.

\begin{remark}
Let us suppose that $X^1, \ldots, X^N$ are conditionally independent,
conditionally on $\mathcal{F}_s^W.$ In this case, $\widehat{\mathcal{X}}%
_{s,t}^{N,i}(X)=\widehat{\mathcal{X}}_{s,t}^{N,i}(X^{i})$ does only depend
on the $i-$th coordinate of the vector $X=(X^{1},\ldots ,X^{N}),$ and the $N$
coordinates of $\widehat{\mathcal{X}}_{s,t}^{N}(X)$ are also conditionally
independent copies of the same dynamic, conditionally on $\mathcal{F}%
_{t}^{W}.$ This is the usual \textbf{synchronous coupling} trick often used
to obtain a quantified rate of convergence in the convergence to mean field
limits. This fact is crucial in our construction.
\end{remark}

We now fix some $n \in {\mathbb{N}} $ and for any $0 \le s < t < \infty, $
we consider the associated dyadic partitions $\pi_n ( s, t) = \{ s < t_{k(s)
+ 1 } < \ldots < t_{k(t) } \le t \} $ as before (recall (\ref{eq:pin})),
where $t_{k}=\frac{k}{2^{n}}.$ We then define 
\begin{multline}
X_{s, t }^{n,N} =\mathcal{X}_{t_{k(t) },t }^{N} \circ \cdots \circ \mathcal{X%
}_{s, t_{k(s) +1} }^{N} ,\quad \overline{X}_{s, t }^{n,N}=\overline{\mathcal{%
X}}_{t_{k(t) },t }^{N} \circ \cdots \circ \overline{\mathcal{X}}_{s, t_{k(s)
+1} }^{N} , \\
\widehat{X}_{s, t}^{n,N}=\widehat{\mathcal{X}}_{t_{k(t) },t }^N \circ \cdots
\circ \widehat{\mathcal{X}}_{s, t_{k(s) +1} }^N .  \label{pf13}
\end{multline}

Our main result now reads as follows.

\begin{theorem}
\label{theo:main} \textbf{A (Euler) }Suppose that Assumptions \ref{ass:H}, %
\ref{ass:perturbation} and \ref{ass:perturbationbis} hold. Let $0\leq s\leq
t\leq T,$ let $X_{s}^{1},...,X_{s}^{N}\in L_{s}^{2} ( {\mathbb{R}}^d) $ be
i.i.d., conditionally on $\mathcal{F}_{s}^{W},$ and write $%
X_{s}=(X_{s}^{1},\ldots ,X_{s}^{N}).$ Then there exists a constant $Q $ that
depends only on $C_{stb},L_{\varphi }$ and on $\overline{L}, $ such that 
\begin{equation}
\mathcal{W}_{2}^{2}(X_{s,t}^{n,N}(X_{s}),\widehat{X}_{s,t}^{n,N}(X_{s}))\leq
e^{QT^{2}}(\frac{1}{N}+\overline{\varepsilon }_{N}+\widehat{\varepsilon }%
_{N}) .  \label{pf10}
\end{equation}%
\newline
\textbf{B (Flow) }Assume additionally that $\mathbf{A(\beta )}$ of (\ref%
{Abeta}) holds for some $\beta >0,$ for $\mathcal{X}_{s,t}^{N}$ (for $%
\widehat{\mathcal{X}}_{s,t}^{N},$ respectively). Consequently a flow
solution $\mathbb{X}_{s,t}^{N}$ ($\widehat{\mathbb{X}}_{s,t}^{N},$
respectively) exists and is unique. Supposing that $X_{s}^{1},...,X_{s}^{N}%
\in L_{s}^{2}$ are i.i.d., conditionally on $\mathcal{F}_{s}^{W}$, we then
have for each $0 \le s \le t\leq T,$ 
\begin{equation}
\mathcal{W}_{2}^{2}(\mathbb{X}_{s,t}^{N}(X_{s}),\widehat{\mathbb{X}}%
_{s,t}^{N}(X_{s}))\leq e^{QT^{2}}(\frac{1}{N}+\overline{\varepsilon }_{N}+%
\widehat{\varepsilon }_{N}).  \label{pf12}
\end{equation}
\end{theorem}

\begin{remark}
Notice that the constant $Q$ does not depend on the the step $2^{-n}$ of the
dyadic grid which is used in the approximation, and therefore the estimate
for the flow is an immediate consequence of the estimates for the Euler
scheme, just by passing to the limit.
\end{remark}

\begin{proof}
The point \textbf{A}\ will be proved in Section \ref{sec:4true} below. (\ref%
{pf12}) is obtained by combining (\ref{pf10}) with (\ref{sl3}), choosing $n$
such that $2^{-n \beta } \le N^{ - 1 } +\overline{\varepsilon }_{N}+\widehat{%
\varepsilon }_{N}.$
\end{proof}

\bigskip

\section{Proof of Theorem \protect\ref{theo:main}}

\label{sec:4true}

\subsection{Coupling of $\mathcal{X}_{s,t}^{N}(X)$ and $\overline{\mathcal{X}%
}_{s,t}^{N}(\overline{X}) $}

The following result is the key estimation of this section.

\begin{lemma}
\label{lem:44} Suppose that Assumptions \ref{ass:H} and \ref%
{ass:perturbation} hold. Let $0\leq s\leq t$ such that $t-s\leq 1,$ and
suppose that $X=(X^{1},\ldots ,X^{N})$ and $\bar{X}=(\bar{X}^{1},\ldots ,%
\bar{X}^{N})\in L_{s}^{2}({\mathbb{R}}^{d\times N}).$ Then 
\begin{equation}
\mathcal{W}_{2}^{2}(\mathcal{X}_{s,t}^{N}(X),\overline{\mathcal{X}}%
_{s,t}^{N}(\overline{X})) \leq \mathcal{W}_{2}^{2}(X,\overline{X})\left(
1+Q(t-s)\right) +2(t-s)\overline{\varepsilon }_N ,  \label{clt3}
\end{equation}
where 
\begin{equation*}
Q =C_{stb}(1+L_{\varphi }^{2})+2\overline{L}.
\end{equation*}
\end{lemma}

\begin{proof}
We recall that 
\begin{equation}
\mathcal{X}_{s,t}(X)=X+h_{s,t}(X)+1_{d_{0}}S_{s,t}^{N}(X),\qquad \overline{%
\mathcal{X}}_{s,t}(\overline{X})=\overline{X}+h_{s,t}(\overline{X})+1_{d_{0}}%
\overline{S}_{s,t}^{N}(\overline{X}).  \label{clt3a}
\end{equation}%
\textbf{Step 1.} Let $\pi_0 ( d x \times d \bar x ) \in \mathcal{P }( {%
\mathbb{R}}^{d \times N} \times {\mathbb{R}}^{ d \times N} ) $ be an optimal
coupling of $X$ and $\overline{X}$ such that 
\begin{equation*}
\mathcal{W}_{2}^{2}(X,\overline{X})= \int_{ {\mathbb{R}}^{ d \times N }
\times {\mathbb{R}}^{ d \times N} } | x - \overline x |_N^2 \pi_0 ( d x, d
\bar x ).
\end{equation*}
In what follows, we work conditionally on this choice and consider $x,$ $%
\bar{x}$ to be deterministic. We write $\Omega_0 = {\mathbb{R}}^{d \times N}
\times {\mathbb{R}}^{ d \times N} ,$ together with its Borel sigma-field $%
\mathcal{A}_0 = \mathcal{B}( {\mathbb{R}}^{d \times N} \times {\mathbb{R}}^{
d \times N} ) , $ and we consider the probability measure $P_0 = \pi_0 $ on $%
(\Omega_0, \mathcal{A}_0).$

\noindent \textbf{Step 2} By Corollary 5.22 of \cite{Villani}, we may choose
a measurable version of the mapping 
\begin{equation*}
\omega_0 = ( x, \bar x) \mapsto \pi_1 ( (x, \bar x ), d s \times d \bar s )
\in \mathcal{P }( {\mathbb{R}}^{d_0 } \times {\mathbb{R}}^{d_0} ) 
\end{equation*}
such that for any fixed $\omega_0 = (x, \bar x) , $ $\pi_1 ( (x, \bar x ), d
s \times d \bar s ) = \pi_1 ( \omega_0, d s \times d \bar s ) $ is the
optimal coupling of $S^N_{s,t } (x) $ and $\bar S_{s, t }^N ( \bar x ),$
that is, 
\begin{eqnarray*}
W_2^2 ( S^N_{s,t } (x), \bar S_{s, t }^N ( \bar x)) &=& \int_{ {\mathbb{R}}%
^{ d_0 } \times {\mathbb{R}}^{d_0} } | s - \bar s |^2 \pi_1 ( (x, \bar x ),
d s \times d \bar s ) \\
&=& \int_{ {\mathbb{R}}^{ d_0 } \times {\mathbb{R}}^{d_0} } | s - \bar s |^2
\pi_1 ( \omega_0, d s \times d \bar s ).
\end{eqnarray*}
We write for short 
\begin{equation*}
( s, \bar s ) = \omega_1 \in \Omega_1 = {\mathbb{R}}^{d_0 } \times {\mathbb{R%
}}^{d_0} , 
\end{equation*}
and we write $\mathcal{A}_1 = \mathcal{B }( {\mathbb{R}}^{d_0 } \times {%
\mathbb{R}}^{d_0}). $ On $( \Omega_1, \mathcal{A}_1),$ we thus consider the
probability measure 
\begin{eqnarray*}
P_1 &= & \int_{ {\mathbb{R}}^{ d \times N } \times {\mathbb{R}}^{ d \times
N} } \pi_0 ( dx \times d \bar x ) \pi_1 ( (x, \bar x ), d s \times d \bar s )
\\
&=& \int_{ \Omega_0} \pi_1 ( \omega_0, d \omega_1 ) P_0 ( d \omega_0) .
\end{eqnarray*}
By construction, $S^N_{s,t } (X, \omega ) \sim \mathcal{L }( s | P_1 ) $ and 
$\bar S^N_{s, t } ( \bar X ) \sim \mathcal{L }( \bar s | P_1), $ that is, $%
P_1$ is a coupling of $S^N_{s,t } (X, \omega ) $ and $\bar S_{s, t }^N (
\bar X, \omega).$ Using (\ref{Aiii}) and (\ref{Aiv}), we deduce that 
\begin{equation*}
\int_{ \Omega_1} | s - \bar s |^2 P_1 ( d \omega ) \le 2 (t-s)\overline{%
\varepsilon }_{N}+ 2 \overline{L}(t-s) \mathcal{W}_{2}^{2}(X,\overline{X}) .
\end{equation*}

\noindent \textbf{Step 3} Let us write $h_{s, t} ( x) = (h_{s, t}^1 (x) ,
\ldots, h_{s, t }^N (x) ) \in {\mathbb{R}}^{ d \times N}$ and introduce 
\begin{equation*}
\Omega_2 = {\mathbb{R}}^{ d \times N} \times {\mathbb{R}}^{ d \times N} ,
\end{equation*}
$\mathcal{A}_2 = \mathcal{B }( {\mathbb{R}}^{ d \times N} \times {\mathbb{R}}%
^{ d \times N}), $ $\omega_2 = ( h , \bar h).$ We use once more Corollary
5.22 of \cite{Villani} to choose a measurable version of the mapping 
\begin{equation*}
\omega_0 = ( x, \bar x) \mapsto \pi_2 ( \omega_0, d h \times d \bar h ) =
\pi_2 ( \omega_0, d \omega_2) , 
\end{equation*}
such that for any fixed $\omega_0,$ $\pi_2 ( \omega_0, d \omega_2 ) $ is the
optimal coupling of $h_{s, t} ( x) $ and of $h_{s, t } (\bar x) ,$ that is, 
\begin{equation*}
\int_{ \Omega_2} \pi_2 ( \omega_0, d h \times d \bar h ) | h - \bar h |_N^2 =%
\mathcal{W}_2^2 ( h_{s, t} ( x), h_{s, t } ( \bar x) ).
\end{equation*}

In what follows we work on the product space 
\begin{equation*}
( \Omega_0 \times \Omega_1 \times \Omega_2 , \mathcal{A}_0 \otimes \mathcal{A%
}_1 \otimes \mathcal{A}_2 , P_0 ( d \omega_0) \otimes \pi_1 ( \omega_0, d
\omega_1)\otimes \pi_2 ( \omega_0, d \omega_2)). 
\end{equation*}
We will also write $\omega = ( \omega_0, \omega_1, \omega_2) = ( (x, \bar
x), ( s, \bar s), (h, \bar j ) ) $ for any element of this product space.

Of course, this product space is different from the initial
filtered probability space $(\Omega, \mathcal{F}, (\mathcal{F}_t)_{t \geq 0}
, P). $ In our setting, the two coordinates $x, \bar x $ of $\omega_0$ are
realizations of $X, \bar X $ which are $\mathcal{F}_s-$measurable. So the
former operation $E_s $ (which was performed on the original space $(\Omega, 
\mathcal{F}, (\mathcal{F}_t)_{t \geq 0} , P) $) has to be replaced by $E_0 (
\cdot) = E ( \cdot | \mathcal{A}_0) ,$ the conditional expectation with
respect to $\mathcal{A}_0.$ This is nothing else than conditioning on $( X,
\bar X ) = ( x, \bar x) = \omega_0$ and working under the measures $\pi_1 (
\omega_0, d \omega_1) \otimes \pi_2 ( \omega_0 , d \omega_2).$ In
particular, we see that working conditionally on $\mathcal{A}_0 $ yields
independence of $\omega_1 $ and $\omega_2.$

The operation $\tilde E_0$ is defined accordingly by $\tilde E_0 ( Z ) = Z -
E_0 ( Z)$ for any random variable which is defined on this product space.

On this product space, we define%
\begin{equation*}
X_{s,t}(\omega) =x+h+1_{d_{0}}s ,\qquad \overline{X}_{s,t} (\omega) =%
\overline{x}+\overline h+1_{d_{0}}\overline{s},
\end{equation*}%
which is, by construction, a coupling of $\mathcal{X}_{s,t}(X) $ and of $%
\overline{\mathcal{X}}_{s,t}(\overline{X}).$ It follows that%
\begin{eqnarray}
\mathcal{W}_{2}^{2}(\mathcal{X}_{s,t}(X),\overline{\mathcal{X}}_{s,t}(%
\overline{X})) &\leq& E \left\vert X_{s,t}-\overline{X}_{s,t}\right\vert
_{N}^{2}  \notag  \label{clt4a} \\
& =&E\left\vert x-\overline{x}+h-\bar h\right\vert _{N}^{2}+E\left\vert
1_{d_{0}}s-1_{d_{0}}\overline{s}\right\vert _{N}^{2},
\end{eqnarray}%
the last equality being a consequence of the fact that conditionally on $%
\mathcal{A}_0, $ $s-\overline{s}$ is centered and independent of $x - \bar x
+ h -\bar h. $

We have 
\begin{eqnarray*}
E\left\vert 1_{d_{0}}s-1_{d_{0}}\overline{s}\right\vert _{N}^{2}
&=&W_{2}^{2}(S_{s,t}^{N}(X),\overline{S}_{s,t}^{N}(\overline{X})) \\
&\leq &2 (t-s)\overline{\varepsilon }_{N}+ 2 \overline{L}(t-s)\mathcal{W}%
_{2}^{2}(X,\overline{X}).
\end{eqnarray*}%
Using orthogonality, we get 
\begin{eqnarray*}
&&E\left\vert (x-\overline{x})+(h_{s,t}(x)-h_{s,t}(\overline{x}))\right\vert
_{N}^{2} \\
&=&E\left\vert (x-\overline{x})+E_{0}(h_{s,t}(x)-h_{s,t}(\overline{x}%
))\right\vert _{N}^{2}+E\left\vert \widetilde{E}_{0}(h_{s,t}(x)-h_{s,t}(%
\overline{x}))\right\vert _{N}^{2}.
\end{eqnarray*}%
Hypothesis (\ref{pf1'}) gives%
\[
E\left\vert \widetilde{E}_{0}(h_{s,t}(x)-h_{s,t}(\overline{x}))\right\vert
_{N}^{2}\leq C_{stb}(t-s)(1+L_{\varphi }^{2})E\left\vert x-\overline{x}%
\right\vert _{N}^{2} 
=C_{stb}(t-s)(1+L_{\varphi }^{2})\mathcal{W}_{2}^{2}(X,\overline{X}).
\]
We write now%
\begin{eqnarray*}
&&E\left\vert (x-\overline{x})+E_{0}(h_{s,t}(x)-h_{s,t}(\overline{x}%
))\right\vert _{N}^{2}=E\left\vert x-\overline{x}\right\vert _{N}^{2}+ \\
&&+2\frac{1}{N}\sum_{i=1}^{N}E((x^{i}-\overline{x}%
^{i})E_{0}(h_{s,t}^{i}(x)-h_{s,t}^{i}(\overline{x}))+E\left\vert
E_{0}(h_{s,t}(x)-h_{s,t}(\overline{x}))\right\vert _{N}^{2}.
\end{eqnarray*}%
We wish to upper bound the double product without using Cauchy-Schwarz
inequality. By (\ref{pf1})%
\begin{equation*}
\left\vert E_{0}(h_{s,t}^{i}(x)-h_{s,t}^{i}(\overline{x}))\right\vert \leq
C_{stb}(t-s)(\left\vert x^{i}-\overline{x}^{i}\right\vert +\left\vert \mu
_{N}(x,\varphi )-\mu _{N}(\overline{x},\varphi )\right\vert ),
\end{equation*}%
so that 
\begin{eqnarray*}
&&E((x^{i}-\overline{x}^{i})E_{0}(h_{s,t}^{i}(x)-h_{s,t}^{i}(\overline{x}))
\\
&\leq &C_{stb}(t-s)E(\left\vert x^{i}-\overline{x}^{i}\right\vert
(\left\vert x^{i}-\overline{x}^{i}\right\vert +\left\vert \mu _{N}(x,\varphi
)-\mu _{N}(\overline{x},\varphi )\right\vert ) \\
&\leq &C_{stb}(t-s)(E\left\vert x^{i}-\overline{x}^{i}\right\vert
^{2}+E(\left\vert x^{i}-\overline{x}^{i}\right\vert \left\vert \mu
_{N}(x,\varphi )-\mu _{N}(\overline{x},\varphi )\right\vert )) \\
&\leq &C_{stb}(t-s)(\frac{3}{2}E\left\vert x^{i}-\overline{x}^{i}\right\vert
^{2}+\frac{1}{2}E(\left\vert \mu _{N}(x,\varphi )-\mu _{N}(\overline{x}%
,\varphi )\right\vert ^{2})),
\end{eqnarray*}%
where we used that $2ab\leq a^{2}+b^{2}.$ But 
\begin{equation*}
E(\left\vert \mu _{N}(x,\varphi )-\mu _{N}(\overline{x},\varphi )\right\vert
^{2})=E\left( \left\vert \frac{1}{N}\sum_{i=1}^{N}\varphi (x_{i})-\varphi (%
\overline{x}_{i})\right\vert ^{2}\right) \leq L_{\varphi }^{2}E\left\vert x-%
\overline{x}\right\vert _{N}^{2}.
\end{equation*}%
Changing the constant if necessary, we conclude that 
\begin{equation*}
2\frac{1}{N}\sum_{i=1}^{N}E((x^{i}-\overline{x}%
^{i})E_{0}(h_{s,t}^{i}(x)-h_{s,t}^{i}(\overline{x}))\leq
C_{stb}(1+L_{\varphi }^{2})(t-s)E\left\vert x-\overline{x}\right\vert
_{N}^{2},
\end{equation*}%
such that 
\begin{eqnarray*}
&&E\left\vert (x-\overline{x})+E_{0}(h_{s,t}(x)-h_{s,t}(\overline{x}%
))\right\vert _{N}^{2} \\
&\leq &E\left\vert x-\overline{x}\right\vert _{N}^{2}+C_{stb}(1+L_{\varphi
}^{2})((t-s)E\left\vert x-\overline{x}\right\vert _{N}^{2}+E\left\vert
E_{0}(h_{s,t}(x)-h_{s,t}(\overline{x}))\right\vert _{N}^{2}.
\end{eqnarray*}%
Since%
\begin{equation*}
E\left\vert E_{0}(h_{s,t}(x)-h_{s,t}(\overline{x}))\right\vert _{N}^{2}\leq
C_{stb}(1+L_{\varphi }^{2})(t-s)^{2}E\left\vert x-\overline{x}\right\vert
_{N}^{2}
\end{equation*}%
and $t-s\leq 1,$ finally we get%
\begin{eqnarray*}
E\left\vert (x-\overline{x})+E_{0}(h_{s,t}(x)-h_{s,t}(\overline{x}%
))\right\vert _{N}^{2} &\leq &E\left\vert x-\overline{x}\right\vert
_{N}^{2}[1+C_{stb}(1+L_{\varphi }^{2})((t-s)] \\
&=&\mathcal{W}_{2}^{2}(X,\overline{X})[1+C_{stb}(1+L_{\varphi }^{2})((t-s)].
\end{eqnarray*}%
Putting things together we deduce that 
$$
E\left\vert x_{s,t}(x,\overline{x})-\overline{x}_{s,t}(x,\overline{x}%
)\right\vert _{N}^{2} \leq \mathcal{W}_{2}^{2}(X,\overline{X})\left(
1+[C_{stb}(1+L_{\varphi }^{2})+2\overline{L}](t-s)\right) \ +2(t-s)\overline{\varepsilon }_{N},
$$
whence the assertion.
\end{proof}

Recall that we consider the partition $\pi _{n}(s,t)$ (see (\ref{eq:pin}).
We suppose it is given by the elements $s=t_{0}<...<t_{l}=t.$ Dropping the
superscript $n,$ we associate to it the chains%
\begin{equation*}
X_{t_{k+1}}^{N}=\mathcal{X}_{t_{k},t_{k+1}}^{N}(X_{t_{k}}) \mbox{ and } 
\overline{X}_{t_{k+1}}^{N}=\overline{\mathcal{X}}_{t_{k},t_{k+1}}^{N}(%
\overline{X}_{t_{k}}),
\end{equation*}%
for any $0\leq k<l$ (see also (\ref{pf13})), with possibly different initial
configurations $X_{s}$ and $\bar{X}_{s},$ both belonging to $L_{s}^{2}.$
Then, using recursively the estimate (\ref{clt3}), we obtain

\begin{proposition}
\label{prop:Euler} Suppose that Assumptions \ref{ass:H} and \ref%
{ass:perturbation} hold. Then for all $k\leq l,$ 
\begin{equation}
\mathcal{W}_{2}^{2}(X_{t_{k}}^{N},\overline{X}_{t_{k}}^{N})\leq
e^{Q(t_{k}-t_{0})}(\mathcal{W}_{2}^{2}(X_{t_{0}}^{N},\overline{X}%
_{t_{0}}^{N})+2 (t_{k}-t_{0})\overline{\varepsilon }_{N})  \label{clt4}
\end{equation}%
with $Q$ a constant which depends on $\overline{L},L_{\varphi }$ and $%
C_{stb} $ only.
\end{proposition}

\subsection{Analysis of the Euler scheme}

We now investigate the terms $h_{s,t}.$ As an immediate consequence of (\ref%
{pf1}), we have for every $X,Y\in L_{s}^{2}({\mathbb{R}}^{d\times N}),$ $s<t,
$ 
\begin{eqnarray}  \label{eq:Hiii}
E\left\vert E_{s}(h_{s,t}(X)-h_{s,t}(Y))\right\vert _{N}^{2} &\leq
&C_{stb}^{2}(1+L_{\varphi }^{2})(t-s)^{2}E\left\vert X-Y\right\vert _{N}^{2},
\notag \\
E\left\vert \widetilde{E}_{s}(h_{s,t}(X)-h_{s,t}(Y))\right\vert _{N}^{2}
&\leq &C_{stb}^{2}(1+L_{\varphi }^{2})(t-s)E\left\vert X-Y\right\vert
_{N}^{2}.
\end{eqnarray}%
The conditional law of large numbers implies the following statement.

\begin{proposition}
Let $X = (X^1, \ldots, X^N ) \in L^2_s ( {\mathbb{R}}^{d \times N} )$ such
that $X^1, \ldots, X^N $ are i.i.d. conditionally on $\mathcal{F}_s^W.$ Then 
\begin{eqnarray}  \label{eq:Hiv}
E\left\vert E_{s}(h_{s,t}(X)-\widehat{h}_{s,t}(X))\right\vert _{N}^{2} &\leq
&\frac{C^2_{stb}(t-s)^{2}\left\Vert \varphi \right\Vert _{\infty }^{2}}{N } ,
\notag \\
E\left\vert \widetilde{E}_{s}(h_{s,t}(X)-\widehat{h}_{s,t}(X))\right\vert
_{N}^{2} &\leq &\frac{C^2_{stb}(t-s)\left\Vert \varphi \right\Vert _{\infty
}^{2}}{N} .
\end{eqnarray}
\end{proposition}

\begin{proof}
Indeed, if $\mu_s = \mathcal{L} ( X^1 | \mathcal{F}_s^W )$ is the commun law
of $X^{1},...,X^{N},$ conditionally on $\mathcal{F}^W_s, $ then $\varphi
(X^{i})-\int \varphi d\mu_s ,i=1,...,N,$ are orthogonal, conditionally on $%
\mathcal{F}^W_s, $ so that\ 
\begin{eqnarray*}
&& E\left\vert E_{s}(h_{s,t}(X)-\widehat{h}_{s,t}(X))\right\vert _{N}^{2} \\
&&=\frac{1}{N}\sum_{i=1}^{N}E\left\vert E_{s}(H_{s,t}^{i}(X^{i},\mu _{N}(X,
\varphi ))-H_{s,t}^{i}(X^{i},\int \varphi d\mu_s ))\right\vert ^{2} \\
&&\leq \frac{C^2_{stb}(t-s)^{2}}{N}\sum_{i=1}^{N}E\left\vert \mu_{N}(X,
\varphi )-\int \varphi d\mu_s )\right\vert ^{2} \leq \frac{%
C^2_{stb}(t-s)^{2}\left\Vert \varphi \right\Vert _{\infty }^{2}}{N } .
\end{eqnarray*}%
The argument for $E\left\vert \widetilde{E}_{s}(h_{s,t}(X)-\widehat{h}%
_{s,t}(X))\right\vert ^{2}$ is analogous.
\end{proof}

\begin{proposition}
\label{concatenation}\textbf{A }Suppose that Assumption \ref{ass:H} holds.\
Fix some $k < l $ and let $X_{j},Y_{j}\in L_{t_{j}}^{2}({\mathbb{R}}%
^{d\times N}),$ for all $j=0,...,k.$ Then%
\begin{eqnarray}
&& E\left\vert
\sum_{j=0}^{k}(h_{t_{j},t_{j+1}}(X_{j})-h_{t_{j},t_{j+1}}(Y_{j}))\right\vert
_{N}^{2}   \notag \\
&& \quad \quad \quad \le 2C^2_{stb}(1+L^2_{\varphi
})(1+(t_{k}-t_{0}))\sum_{j=0}^{k}E\left\vert X_{j}-Y_{j}\right\vert
_{N}^{2}(t_{j+1}-t_{j}).  \label{AE5}
\end{eqnarray}%
\textbf{B }If moreover for every $j=0,...,k,$ the components $%
X_{j}^{1},...,X_{j}^{N}$ are i.i.d., conditionally on $\mathcal{F}_{t_j}^W,$
then 
\begin{equation}
E\left\vert \sum_{j=0}^{k}(h_{t_{j},t_{j+1}}(X_{j})-\widehat{h}%
_{t_{j},t_{j+1}}(X_{j}))\right\vert _{N}^{2}\leq \frac{%
2C^2_{stb}(t_{k}-t_{0})(1+t_{k}-t_{0})\left\Vert \varphi \right\Vert
_{\infty }^{2}}{N }.  \label{AE6}
\end{equation}
\end{proposition}

\begin{proof}
\textbf{A } We denote $\gamma _{i}=t_{i+1}-t_{i}, $ $i \geq 0 ,$ and we
write 
\begin{equation*}
E\left\vert
\sum_{j=0}^{k}(h_{t_{j},t_{j+1}}(X_{j})-h_{t_{j},t_{j+1}}(Y_{j})\right\vert
_{N}^{2}\leq 2\Sigma _{1}+2\Sigma _{2},
\end{equation*}
with%
\begin{eqnarray*}
\Sigma _{1} &=&E\left\vert
\sum_{j=0}^{k}E_{t_{j}}(h_{t_{j},t_{j+1}}(X_{j})-h_{t_{j},t_{j+1}}(Y_{j}))%
\right\vert _{N}^{2} \\
&=&E\left\vert \sum_{j=0}^{k}\gamma _{j}\frac{%
E_{t_{j}}(h_{t_{j},t_{j+1}}(X_{j})-h_{t_{j},t_{j+1}}(Y_{j}))}{\gamma _{j}}%
\right\vert _{N}^{2} \\
&\leq &(t_{k+1}-t_{0})\sum_{j=0}^{k}\gamma _{j}E\left\vert \frac{%
E_{t_{j}}(h_{t_{j},t_{j+1}}(X_{j})-h_{t_{j},t_{j+1}}(Y_{j}))}{\gamma _{j}}%
\right\vert _{N}^{2} \\
&\leq &C^2_{stb}(1+L_{\varphi }^{2})(t_{k+1}-t_{0})\sum_{j=0}^{k}\gamma
_{j}E\left\vert X_{j}-Y_{j}\right\vert _{N}^{2},
\end{eqnarray*}%
where we have used (\ref{eq:Hiii}). Moreover, since the $\widetilde{E}%
_{t_{j}}(h_{t_{j},t_{j+1}}(X_{j})-h_{t_{j},t_{j+1}}(Y_{j})),j=0,...,k,$ are
orthogonal, and using (\ref{eq:Hiii}) once more, 
\begin{eqnarray*}
\Sigma _{2} &=&E\left\vert \sum_{j=0}^{k}\widetilde{E}%
_{t_{j}}(h_{t_{j},t_{j+1}}(X_{j})-h_{t_{j},t_{j+1}}(Y_{j}))\right\vert
_{N}^{2} \\
&=&\sum_{j=0}^{k}E\left\vert \widetilde{E}%
_{t_{j}}(h_{t_{j},t_{j+1}}(X_{j})-h_{t_{j},t_{j+1}}(Y_{j}))\right\vert
_{N}^{2} \leq C^2_{stb}(1+L_{\varphi }^{2})\sum_{j=0}^{k}\gamma
_{j}E\left\vert X_{j}-Y_{j}\right\vert _{N}^{2}.
\end{eqnarray*}

\noindent \textbf{B }As above we write $E\left\vert
\sum_{j=0}^{k}(h_{t_{j},t_{j+1}}(X_{j})-\widehat{h}_{t_{j},t_{j+1}}(X_{j})%
\right\vert _{N}^{2}\leq 2\Sigma _{1}+2\Sigma _{2} $ with%
\begin{eqnarray*}
\Sigma _{1} &=&E\left\vert \sum_{j=0}^{k}E_{t_{j}}(h_{t_{j},t_{j+1}}(X_{j})-%
\widehat{h}_{t_{j},t_{j+1}}(X_{j}))\right\vert _{N}^{2} \\
&\leq &(t_{k+1}-t_{0})\sum_{j=0}^{k}\gamma _{j}E\left\vert \frac{%
E_{t_{j}}(h_{t_{j},t_{j+1}}(X_{j})-\widehat{h}_{t_{j},t_{j+1}}(X_{j}))}{%
\gamma _{j}}\right\vert _{N}^{2} \\
& \leq & \frac{C^2_{stb}(t_{k+1}-t_{0})^{2}\left\Vert \varphi \right\Vert
_{\infty }^{2}}{N} ,
\end{eqnarray*}%
where we have used (\ref{eq:Hiv}) to obtain the last line. Moreover, 
\begin{eqnarray*}
\Sigma _{2} &=&E\left\vert \sum_{j=0}^{k}\widetilde{E}%
_{t_{j}}(h_{t_{j},t_{j+1}}(X_{j})-\widehat{h}_{t_{j},t_{j+1}}(X_{j}))\right%
\vert _{N}^{2} \\
&=&\sum_{j=0}^{k}E\left\vert \widetilde{E}_{t_{j}}(h_{t_{j},t_{j+1}}(X_{j})-%
\widehat{h}_{t_{j},t_{j+1}}(X_{j}))\right\vert _{N}^{2} \leq \frac{%
C^2_{stb}(t_{k+1}-t_{0})\left\Vert \varphi \right\Vert _{\infty }^{2}}{N},
\end{eqnarray*}%
using (\ref{eq:Hiv}) once more.
\end{proof}

\subsection{Conditional Law of Large numbers}

\label{LN8} We come back to the setting of Theorem \ref{theo:main}. Recall
that for any $s < t \le T $ fixed we consider the partition $\pi_n ( s, t) =
\{s= t_{0}<t_{1}<...<t_{l} = t \}$ (see (\ref{eq:pin})). As before, we write 
$\gamma _{i}=t_{i+1}-t_{i},$ for any $i \geq 0.$ Given initial conditions $%
X_{s}^{i}, \overline{X}_s^i, \widehat{X}_s^i , 1 \le i \le N , $ all
belonging to $L^2_s ,$ in accordance with (\ref{pf7})--(\ref{pf13}), we
define recursively for any $0 \le k < l , $ 
\begin{equation}
X_{t_{k+1}}^{N}=\mathcal{X}_{t_{k},t_{k+1}}^{N}(X_{t_{k}}^{ N}),\quad 
\overline{X}_{t_{k+1}}^{N}=\overline{\mathcal{X}}_{t_{k},t_{k+1}}^{N}(%
\overline{X}_{t_{k}}^{ N}),\qquad \widehat{X}_{t_{k+1}}^{N}=\widehat{%
\mathcal{X}}_{t_{k},t_{k+1}}^{N}(\widehat{X}_{t_{k}}^{ N}).  \label{pf13bis}
\end{equation}

\begin{proposition}
Grant Assumptions \ref{ass:H}, \ref{ass:perturbation} and \ref%
{ass:perturbationbis} and suppose moreover that $\widehat{X}_s^1, \ldots, 
\widehat{X}_s^N $ are i.i.d., conditionally on $\mathcal{F}_{{s}}^W. $ Then
for all $k\leq l,$ the $N$ coordinates of $\widehat{X}_{t_{k}}^{N}$ are
i.i.d., conditionally on $\mathcal{F}_{t_{k}}^{W}.$
\end{proposition}

\begin{proof}
If $\widehat{X}_{s}^{1},\ldots ,\widehat{X}_{s}^{N}$ are i.i.d.,
conditionally on $\mathcal{F}_{{s}}^{W}$ then $\widehat{\mathcal{X}}%
_{s,t}^{N,i}(X),i=1,...,N,$ are also i.i.d., conditionally on $\mathcal{F}_{{%
s}}^{W}.$ So the proof follows by induction over $k,$ by construction of the
scheme.
\end{proof}

From now on we will impose the following assumption on the initial
configurations.

\begin{assumption}
\label{ass:iid}

\begin{enumerate}
\item $X_{s}^{1},...,X_{s}^{N} $ belong to $L^2_s, $ and they are i.i.d.,
conditionally on $\mathcal{F}_{{s}}^W. $

\item $\overline{X}_s^1, \ldots, \overline{X}_s^N $ belong to $L^2_s, $ and
they are i.i.d., conditionally on $\mathcal{F}_{{s}}^W. $

\item $\widehat{X}_s^1, \ldots, \widehat{X}_s^N$ belong to $L^2_s, $ and
they are i.i.d., conditionally on $\mathcal{F}_{{s}}^W. $
\end{enumerate}
\end{assumption}

We will denote by $\mu _{t_k}^W= \mathcal{L} ( \widehat{X}_{t_{k}}^{N, i} | 
\mathcal{F}_{t_k}^W) $ the commun law of $\widehat{X}_{t_{k}}^{N,
i},i=1,...,N,$ conditionally on $\mathcal{F}_{t_k}^W.$

In order to compare $\widehat{X}^{N}$ with $\overline{X}^{N},$ we write%
\begin{equation*}
\widehat{X}_{t_{k+1}}^{N,i}=\widehat{X}_{t_{k}}^{N,i}+H_{t_{k},t_{k+1}}^{i}(%
\widehat{X}_{t_{k}}^{N,i},\mu _{N}(\widehat{X}_{t_{k}}^{N},\varphi
))+1_{d_{0}}\bar{S}_{t_{k},t_{k+1}}^{N}(\widehat{X}%
_{t_{k}}^{N})+I_{k}^{i}+J_{k},
\end{equation*}%
where, recalling (\ref{pf3}) and (\ref{Avibis}), 
\begin{eqnarray}
I_{k}^{i} &=&H_{t_{k},t_{k+1}}^{i}(\widehat{X}_{t_{k}}^{N,i},\int \varphi
(z)\mu _{t_{k}}^{W}(dz))-H_{t_{k},t_{k+1}}^{i}(\widehat{X}_{t_{k}}^{N,i},\mu
_{N}(\widehat{X}_{t_{k}}^{N},\varphi ))  \label{LN1} \\
&=&\widehat{h}_{t_{k},t_{k+1}}^{i}(\widehat{X}%
_{t_{k}}^{N})-h_{t_{k},t_{k+1}}^{i}(\widehat{X}_{t_{k}}^{N}),  \notag \\
J_{k} &=&1_{d_{0}}\widehat{S}_{t_{k},t_{k+1}}^{N}(\widehat{X}%
_{t_{k}}^{N})-1_{d_{0}}\bar{S}_{t_{k},t_{k+1}}^{N}(\widehat{X}_{t_{k}}^{N}).
\label{LN2}
\end{eqnarray}

\begin{proposition}
Grant Assumptions \ref{ass:H}, \ref{ass:perturbationbis} and \ref{ass:iid}.
Then it holds that 
\begin{equation}
E\left\vert \sum_{j=0}^{k}I_{j}\right\vert _{N}^{2}\leq \frac{%
2C_{stb}^{2}(1+(t_{k+1}-t_{0}))(t_{k+1}-t_{0})\left\Vert \varphi \right\Vert
_{\infty }^{2}}{N}  \label{LN3}
\end{equation}%
and 
\begin{equation}
E\left\vert \sum_{j=0}^{k}J_{j}\right\vert ^{2}\leq (t_{k+1}-t_{0})\widehat{%
\varepsilon }_{N}.  \label{LN4}
\end{equation}
\end{proposition}

\begin{proof}
(\ref{LN3}) follows from (\ref{AE6}). We estimate now $E\left\vert
\sum_{j=0}^{k}J_{j}\right\vert ^{2}.$ Notice that for $i<j,$ we have 
\begin{equation*}
E(J_{i}J_{j})=E(J_{i}E(J_{j}\mid \mathcal{F}_{t_{i+1}}))=0,
\end{equation*}
and, by (\ref{Av}), $E\left\vert J_{j}\right\vert ^{2}\leq (t_{j+1}-t_{j})%
\widehat{\varepsilon }_{N}.$ Therefore, 
\begin{equation*}
E\left\vert \sum_{j=0}^{k}J_{j}\right\vert ^{2}=\sum_{j=0}^{k}E\left\vert
J_{j}\right\vert ^{2}\leq (t_{k+1}-t_{0})\widehat{\varepsilon }_{N}.
\end{equation*}
\end{proof}

We are now able to resume our findings in the following

\begin{theorem}
\label{LLNW2} Grant Assumptions \ref{ass:H}, \ref{ass:perturbationbis} and %
\ref{ass:iid} . Then for any $k<l,$ 
\begin{equation}
E\left\vert \overline{X}_{t_{k+1}}^{N}-\widehat{X}_{t_{k+1}}^{N}\right\vert
_{N}^{2}\leq 4e^{Q_{1}(t_{k+1}+1)^{2}}\left\{ E\left\vert \overline{X}%
_{t_{0}}^{N}-\widehat{X}_{t_{0}}^{N}\right\vert _{N}^{2}+Q_{2}(\frac{1}{N}+%
\widehat{\varepsilon }_{N})(t_{k+1}-t_{0})\right\} ,  \label{clt5}
\end{equation}%
where $Q_{1}$ and $Q_{2}$ are constants depending on $T, C_{stb},L_{\varphi }
$ and $\overline{L}$ only.%
%
%
%
%
%
%
\end{theorem}

\begin{proof}
In what follows, to ease notation, we drop the superscript $N.$ We write 
\begin{eqnarray*}
\widehat{X}_{t_{k+1}}^{i}-\overline{X}_{t_{k+1}}^{i} &=&\widehat{X}%
_{t_{0}}^{i}-\overline{X}_{t_{0}}^{i}+\sum_{j=0}^{k}(h_{t_{j},t_{j+1}}^{i}(%
\widehat{X}_{t_{j}}^{i})-h_{t_{j},t_{j+1}}^{i}(\overline{X}_{t_{j}}^{i})) \\
&&+\sum_{j=0}^{k}1_{d_{0}}\bar{S}_{t_{j},t_{j+1}}^{N}(\widehat{X}%
_{t_{j}}^{N})-1_{d_{0}}\bar{S}_{t_{j},t_{j+1}}^{N}(\overline{X}%
_{t_{j}}^{N})+\sum_{j=0}^{k}I_{j}^{i}+\sum_{j=0}^{k}J_{j},
\end{eqnarray*}%
where we use the same notation as in (\ref{LN1}) and (\ref{LN2}). We have
proved in (\ref{AE5}) that 
\begin{multline*}
E\left\vert \sum_{j=0}^{k}(h_{t_{j},t_{j+1}}^{i}(\widehat{X}%
_{t_{j}})-h_{t_{j},t_{j+1}}^{i}(\overline{X}_{t_{j}}))\right\vert
_{N}^{2}\leq \\
\leq 2C_{stb}^{2}(1+L_{\varphi }^{2})(1+t_{k+1}-t_{0}))\sum_{j=0}^{k}E\left\vert 
\overline{X}_{t_{j}}-\widehat{X}_{t_{j}}\right\vert _{N}^{2}(t_{j+1}-t_{j}).
\end{multline*}

Moreover, by item i) of Assumption \ref{ass:perturbation}, $\bar{S}%
_{t_{j},t_{j+1}}^{N}(\widehat{X}_{t_{j}}^{N})- \bar{S}_{t_{j},t_{j+1}}^{N}(%
\overline{X}_{t_{j}}^{N}),j=1,...,k,$ are orthogonal. Therefore, recalling
also (\ref{Aiii}), 
\begin{multline*}
E\left\vert \sum_{j=0}^{k}1_{d_{0}}\bar{S}_{t_{j},t_{j+1}}^{N}(\widehat{X}%
_{t_{j}}^{N})-1_{d_{0}}\bar{S}_{t_{j},t_{j+1}}^{N}(\overline{X}%
_{t_{j}}^{N})\right\vert ^{2} \\
=\sum_{j=0}^{k}E\left\vert \bar{S}_{t_{j},t_{j+1}}^{N}(\widehat{X}%
_{t_{j}}^{N})-\bar{S}_{t_{j},t_{j+1}}^{N}(\overline{X}_{t_{j}}^{N})\right%
\vert ^{2} 
\leq \overline{L}\sum_{j=0}^{k}E\left\vert \widehat{X}_{t_{j}}^{N}-\overline{%
X}_{t_{j}}^{N}\right\vert_N^{2}(t_{j+1}-t_{j}).
\end{multline*}
Gathering these results, we obtain that 
\begin{multline*}
E\left\vert \overline{X}_{t_{k+1}}-\widehat{X}_{t_{k+1}}\right\vert
_{N}^{2}\leq 4E\left\vert \overline{X}_{t_{0}}-\widehat{X}%
_{t_{0}}\right\vert _{N}^{2}+C_{1}\sum_{j=0}^{k}E\left\vert \overline{X}%
_{t_{j}}-\widehat{X}_{t_{j}}\right\vert _{N}^{2}(t_{j+1}-t_{j}) \\
+C_{2}(t_{k+1}-t_{0})(N^{-1}+\widehat{\varepsilon }_{N}),
\end{multline*}%
for some constants $C_1, C_2$ depending on $T, C_{stb},L_{\varphi }$ and $%
\overline{L}$ only. Using the discrete version of the Gronwall lemma allows
to conclude.


%
%
\end{proof}

We are now ready to give the

\begin{proof}[Proof of Theorem \protect\ref{theo:main}]
The assertion of item \textbf{A} is a consequence of (\ref{clt4}) and (\ref%
{clt5}).
\end{proof}

\section{The Central Limit Theorem example}

\label{sec:CLT} Here we present a general framework where the coupling
property (\ref{Aiv}) is a consequence of a quantified version of the Central
Limit Theorem (CTL). We work in the general framework of Section \ref%
{sec:frameworkgeneral}. To model the perturbations, here we consider the
following model. Let $W$ be a $d_0-$dimensional standard Brownian motion and
put $W_{s, t} = W_t - W_s.$ Moreover we put \newline
$\clubsuit \qquad $ 
\begin{equation}
S_{s,t}^{N}(x) =\frac{1}{\sqrt{N}}\sum_{i=1}^{N}v_{s,t}^{i}(\omega
,x^{i})\in {\mathbb{R}}^{d_{0}},\quad x=(x^{1},...,x^{N})\in {\mathbb{R}}%
^{d\times N},  \label{truepf4}
\end{equation}%
where for each $1\leq i\leq N,$ 
\begin{equation*}
\Omega \times {\mathbb{R}}^{d}\ni (\omega ,x^{i})\mapsto v_{s,t}^{i}(\omega
,x^{i})\in {\mathbb{R}}^{d_{0}}
\end{equation*}%
is $\mathcal{G}_{s,t}^{2}\otimes \mathcal{B}({\mathbb{R}}^{d})-\mathcal{B}({%
\mathbb{R}}^{d_{0}})-$ measurable, and where for each fixed $x^i , $ 
\begin{equation*}
E|v_{s,t}^{i}(\omega ,x^{i})|^{2}<\infty .
\end{equation*}%
We will write for short $v_{s, t}^i ( x^i ) = v_{s,t}^{i}(\omega ,x^{i}), $
and we assume that for every $x=(x^{1},...,x^{N})\in {\mathbb{R}}^{d\times
N},$ the random variables $v_{s,t}^{1}(x^{1}),$ $\ldots ,v_{s,t}^{N}(x^{N})$
are independent and centered. Since they are $\mathcal{G}_{s,t}^{2}\times 
\emph{B}({\mathbb{R}}^{d})$\ measurable, they are independent of the past $%
\mathcal{F}_{s},$ independent of the variables $H_{s,t}^{j}(x^{j},a)$
introduced in Section \ref{sec:frameworkgeneral} above, and independent of $%
W.$ We also assume that there are some numbers $M_{2}$ and $M_{4}$ such
that, for every $i$ and every $x^{i},$ 
\begin{equation}
E\left\vert v_{s,t}^{i}(\omega ,x^{i})\right\vert ^{2}\leq M_{2}(t-s),\qquad
E\left\vert v_{s,t}^{i}(\omega ,x^{i})\right\vert ^{4}\leq M_{4}(t-s)^{2}.
\label{MD0}
\end{equation}

Finally we assume that there exists a function $f:{\mathbb{R}}%
^{d}\rightarrow {\mathbb{R}}^{d_{0}\times d_{0}},$ taking values in the set
of all symmetric positive-definite $d_{0}\times d_{0}-$matrices, such that
for every $i=1,...,N$ and $j,p\in \{1,...,d_{0}\},$ 
\begin{equation}
E((v_{s,t}^{i}(x^{i}))^{j}(v_{s,t}^{i}(x^{i}))^{p})=Cov((v_{s,t}^{i}(x^{i}))^{j},(v_{s,t}^{i}(x^{i}))^{p})=f^{j,p}(x^{i})(t-s).
\label{MD1}
\end{equation}%
Here, $f^{j,p}$ denotes the $(j,p)-$entry of the matrix valued function $f.$
We stress that $f$ does not depend on $i.$

\begin{example}
Let $M^{i}(dt,dz,du),1\leq i\leq N,$ be i.i.d. Poisson random measures on ${%
\mathbb{R}}_{+}\times {\mathbb{R}}_{+}\times {\mathbb{R}}^{d_{0}}$ having
intensity $dtdz\pi (du)$ such that $\pi $ is centered with 
\begin{equation*}
\int_{{\mathbb{R}}^{d_0}}u^{j}u^{p}\pi (du)=\sigma ^{j,p},1\leq j,p\leq
d_{0}.
\end{equation*}%
If we put 
\begin{equation*}
v_{s,t}^{i}(\omega ,x^{i})=\int_{]s,t]}\int_{{\mathbb{R}}_{+}}\int_{{\mathbb{%
R}}^{d_0}}u1_{\{z\leq \gamma (x^{i})\}}M^{i}(dz,du,dr),
\end{equation*}%
then $f(x^{i})=\gamma (x^{i})\sigma .$
\end{example}

We come now back to the function $f$ defined in (\ref{MD1}). We assume that $%
f:{\mathbb{R}}^{d}\rightarrow {\mathbb{R}}^{d_0\times d_0}$ is Lipschitz
continuous; e.g. for all $z,z^{\prime }\in {\mathbb{R}}^{d},$ 
\begin{equation*}
\Vert f(z)-f(z^{\prime })\Vert _{op}\leq L_{f}|z-z^{\prime }|.
\end{equation*}%
We also suppose that $f$ is uniformly positive definite, that is, there
exists $f_{\ast }>0$ such that for every $x\in {\mathbb{R}}^{d},$%
\begin{equation*}
\inf_{y\in {\mathbb{R}}^{d_0}:\left\vert y\right\vert
=1}\sum_{j,p=1}^{d_{0}}f^{j,p}(x)y^{j}y^{p}\geq f_{\ast }>0.
\end{equation*}%
Finally, we assume there exists $f^{\ast }<\infty $ such that for all $x\in {%
\mathbb{R}}^{d},$ 
\begin{equation*}
\sup_{y\in {\mathbb{R}}^{d_0}:\left\vert y\right\vert
=1}\sum_{j,p=1}^{d_{0}}f^{j,p}(x)y^{j}y^{p}\leq f^{\ast }<\infty .
\end{equation*}%
If all these conditions are verified, we say that $\mathbf{A(f)}$ holds.

We now introduce $F_{N}:{\mathbb{R}}^{d\times N}\rightarrow {\mathbb{R}}%
^{d_0\times d_{0}}$ given by 
\begin{equation*}
F_{N}(x)=\mu _{N}(x,f)=\frac{1}{N}\sum_{i=1}^{N}f(x^{i})\in {\mathbb{R}}%
^{d_0\times d_{0}},\qquad x=(x^{1},...,x^{N})\in {\mathbb{R}}^{d\times N}.
\end{equation*}%
Clearly, 
\begin{equation*}
E(S_{s,t}^{N}(x,f)^{j}S_{s,t}^{N}(x,f)^{p})=F_{N}^{j,p}(x)(t-s).
\end{equation*}%
Notice that $F_{N}$ is bounded and Lipschitz continuous, since we consider
the norm $\left\vert x\right\vert _{N}^{2}=\frac{1}{N}\sum_{i=1}^{N}\left%
\vert x^{i}\right\vert ^{2}$ on ${\mathbb{R}}^{d\times N}.$ Moreover for
every $N$ and every $x\in {\mathbb{R}}^{d\times N},$%
\begin{equation*}
\sup_{y\in {\mathbb{R}}^{d_0},\left\vert y\right\vert
=1}\sum_{j,p=1}^{d_{0}}F_{N}^{j,p}(x)y^{j}y^{p}\leq f^{\ast }
\end{equation*}%
and also 
\begin{equation}
\inf_{y\in {\mathbb{R}}^{d_0},\left\vert y\right\vert
=1}\sum_{j,p=1}^{d_{0}}F_{N}^{j,p}(x)y^{j}y^{p}\geq f_{\ast }>0.
\label{sigmaL}
\end{equation}

As a consequence, for any $x\in {\mathbb{R}}^{d\times N},$ there exists a
unique $d_{0}\times d_{0}-$matrix 
\begin{equation}
\sigma _{N}(x)=\sqrt{F_{N}(x)}  \label{sigmaN}
\end{equation}%
which is symmetric and positive definite such that $\sigma _{N}(x)\sigma
_{N}(x)=F_{N}(x),$ and such that $\sigma _{N}(x)$ is Lipschitz continuous
and bounded as a function of $x.$

We now define 
\begin{equation}
\overline{S}_{s,t}^{N}(x,f)=\sigma _{N}(x)(W_{t}-W_{s})\in {\mathbb{R}}%
^{d_{0}},  \label{pf5}
\end{equation}%
where $W$ is the $d_{0}-$dimensional Brownian motion introduced in the
beginning of this section. We notice that $\overline{S}_{s,t}^{N}(x,f)$ is a
centered $d_{0}-$dimensional Gaussian vector which has the same covariance
matrix $F_{N}(x)$ as $S_{s,t}^{N}(x,f),$ that is, 
\begin{equation*}
E(S_{s,t}^{N}(x,f)^{j}S_{s,t}^{N}(x,f)^{p})=E(\overline{S}_{s,t}^{N}(x,f)^{j}%
\overline{S}_{s,t}^{N}(x,f)^{p})=F_{N}^{j,p}(x).
\end{equation*}%
Finally, we put, for any $X=(X^{1},\ldots ,X^{N})\in L_{s}^{2}({\mathbb{R}}%
^{d\times N}),$ 
\begin{equation}
\bar{F}_{s}^{W}(X)=E(F_{N}(X)|\mathcal{F}_{s}^{W}),
\end{equation}%
an almost surely symmetric and positive definite $d_{0}\times d_{0}-$matrix
satisfying almost surely, 
\begin{equation*}
\inf_{y\in {\mathbb{R}}^{d_0},\left\vert y\right\vert
=1}\sum_{j,p=1}^{d_{0}}(\bar{F}_{s}^{W})^{j,p}y^{j}y^{p}\geq f_{\ast }>0,
\end{equation*}%
such that 
\begin{equation}
\bar{\sigma}_{s}^{W}(X)=\sqrt{\bar{F}_{s}^{W}(X)},\;\bar{\sigma}_{s}^{W}(X)%
\bar{\sigma}_{s}^{X}(X)=\bar{F}_{s}^{W}(X),  \label{barsigmas}
\end{equation}%
is well-defined and defines another almost surely symmetric and positive
definite matrix. Notice that whenever $X^{1},\ldots ,X^{N}$ are i.i.d.,
conditionally on $\mathcal{F}_{s}^{W},$ then 
\begin{equation}
\bar{F}_{s}^{W}(X)=E(f(X^{1})|\mathcal{F}_{s}^{W}).
\end{equation}

Having this in mind, we finally introduce 
\begin{equation}
\widehat{S}_{s,t}(X,f)=\bar{\sigma}_{s}^{W}(X)(W_{t}-W_{s})\in {\mathbb{R}}%
^{d_{0}},
\end{equation}%
which, conditionally on $\mathcal{F}_{s}^{W},$ is a Gaussian vector having
the covariance matrix $\bar{F}_{s}^{W}(X).$

\begin{remark}
Recalling our hypothesis (\ref{Avi}), we see that the representation
property is satisfied if we take $W_{s,t}=W_{t}-W_{s}$ and $\Gamma _{s,t} (
w, a ) = \sqrt{a} w ,$ where for any positive definite symmetric $d_0 \times
d_0-$matrix, $\sqrt{a} $ denotes its matrix square root.
\end{remark}

In what follows, we check that the coupling property (\ref{Aiv}) is
satisfied.

\subsection{A quantified version of the Central Limit Theorem}

We start comparing the law of $S_{s,t}^{N}(x,f)$ and of $\overline{S}%
_{s,t}^{N}(x,f)$ in $W_{2}-$distance. To do so, we introduce 
\begin{equation*}
Z^{i}=\frac{1}{\sqrt{t-s}}\sigma _{N}^{-1}(x)v_{s,t}^{i}(\omega
,x^{i}),1\leq i\leq N.
\end{equation*}%
The covariance matrix of $\frac{1}{\sqrt{N}}\sum_{i=1}^{N}Z^{i}$ is the
identity matrix. Observe that, as a consequence of (\ref{MD0}), for any $%
1\leq i\leq N,$ 
\begin{equation*}
E\left\vert Z^{i}\right\vert ^{2}\leq \frac{M_{2}}{f_{\ast }},\qquad
E\left\vert Z^{i}\right\vert ^{4}\leq \frac{M_{4}}{f_{\ast }^{2}}.
\end{equation*}%
Then, Theorem 5 in the paper \cite{Bonis} of Thomas Bonis implies that 
\begin{equation}
W_{2}^{2}(\frac{1}{\sqrt{N}}\sum_{i=1}^{N}Z^{i},\gamma _{d_{0}})\leq
C_{Bonis}(M_{2},M_{4})\frac{1}{Nf_{\ast }^{4}},  \label{tclbonis}
\end{equation}%
where $C_{Bonis}(M_{2},M_{4})$ is a universal and explicit constant
depending only on the bounds $M_{2}$ and $M_{4},$ and where $\gamma _{d_{0}}=%
\mathcal{N}(0_{d_0},Id_{d_0})$ denotes the $d_{0}-$dimensional standardized
Gaussian measure. Notice that 
\begin{multline*}
W_{2}^{2}(S_{s,t}^{N}(x,f),\overline{S}_{s,t}^{N}(x,f)) \\
\leq (t-s)f^{\ast }W_{2}^{2}(\sigma
_{N}^{-1}(x)(t-s)^{-1/2}S_{s,t}^{N}(x,f),(t-s)^{-1/2}(W_{t}-W_{s})) \\
=(t-s)f^{\ast }W_{2}^{2}(\frac{1}{\sqrt{N}}\sum_{i=1}^{N}Z^{i},\gamma
_{d_{0}}).
\end{multline*}

We summarize the above discussion in the following

\begin{lemma}
Suppose that hypothesis $\mathbf{A(f)}$ holds. Then there exists a universal
constant $C_{Bonis}(M_{2},M_{4})$ such that, for every $x\in {\mathbb{R}}%
^{d\times N}$ and every $s<t,$%
\begin{equation}
W_{2}^{2}(S_{s,t}^{N}(x,f),\overline{S}_{s,t}^{N}(x,f))\leq
C_{Bonis}(M_{2},M_{4})\frac{(t-s)f^{\ast }}{Nf_{\ast }^{4}}.  \label{MD4}
\end{equation}%
As a consequence, condition (\ref{Aiv}) holds with $\overline{\varepsilon }%
_{N}=C_{Bonis}(M_{2},M_{4})\frac{f^{\ast }}{f_{\ast }^{4}} N^{ - 1 }.$
\end{lemma}

\subsection{The Lipschitz property (\protect\ref{Aiii}) and the law of large
numbers (\protect\ref{Av})}

Recall that $\Vert A\Vert _{HS}$ denotes the Hilbert-Schmidt norm of a $%
d_{0}\times d_{0}-$matrix $A$ and that $\sigma _{N}(X)$ is the matrix square
root of $F_{N}(X).$

\begin{proposition}
\label{prop:52} Let $X = (X^1, \ldots, X^N )$ and $\bar X = (\bar X^1,
\ldots, \bar X^N) \in L^2_s ( {\mathbb{R}}^{d \times N} ).$ Then we have
that 
\begin{equation*}
\| \overline{S}_{s,t}^{N}(X,f) - \overline{S}_{s,t}^{N}(\bar X,f)\|_2^2 = E
( \| \sigma_N ( X) - \sigma_N ( \bar X) \|_{HS}^2 ) (t-s) .
\end{equation*}
\end{proposition}

\begin{proof}
We have that $\overline{S}_{s,t}^{N}(X,f) - \overline{S}_{s,t}^{N}(\bar
X,f)= ( \sigma_N (X) - \sigma_N ( \bar X) ) (W_t- W_s) .$ In what follows,
we work conditionally on $\mathcal{F}_s, $ such that $A :=( \sigma_N (X) -
\sigma_N ( \bar X) )$ is a deterministic symmetric matrix and $G_N := A
(W_t- W_s) $ a Gaussian vector with covariance matrix $(t-s) A A .$ So,
conditionally on $\mathcal{F}_s, $ 
\begin{equation*}
E_s( |G_N|^2) = \sum_{i=1}^{d_0} E_s ((G_N^i)^2 ) = (t-s) \sum_{i=1}^{d_0}
\sum_{k=1}^{d_0} (A^{i,k})^2 = (t-s) \| A \|_{HS}^2.
\end{equation*}
Taking expectation then yields the result.
\end{proof}

In what follows we will use the following fact which is a consequence of
results obtained by Yu. B. Farforovskaya in \cite{Farf}.

\begin{proposition}
\label{prop:53} Let $X=(X^{1},\ldots ,X^{N})$ and $\bar{X}=(\bar{X}%
^{1},\ldots ,\bar{X}^{N})\in L_{s}^{2}({\mathbb{R}}^{d\times N}).$ Then
almost surely, 
\begin{equation*}
\Vert \sqrt{F_{N}(X)}-\sqrt{F_{N}(\bar{X})}\Vert _{op}\leq \frac{8(\log
d_{0}+2)^{2}}{\sqrt{f_{\ast }}}\Vert F_{N}(X)-F_{N}(\bar{X})\Vert _{op}.
\end{equation*}%
In particular, recalling that $\Vert A\Vert _{op}^{2}\leq \Vert A\Vert
_{HS}^{2}\leq d_{0}\Vert A\Vert _{op}^{2},$ and using the Lipschitz
continuity of $f,$ 
\begin{eqnarray}
\Vert \overline{S}_{s,t}^{N}(X,f)-\overline{S}_{s,t}^{N}(\bar{X},f)\Vert
_{2}^{2} &\leq &\frac{C_{d_{0}}}{f_{\ast }}L_{f}^{2}(t-s)E|X-\bar{X}%
|_{N}^{2},  \label{eq:ovs} \\
with\quad C_{d_{0}} &=&8^{2}d_{0}(\log d_{0}+2)^{4}.
\end{eqnarray}%
As a consequence, the Lipschitz property (\ref{Aiii}) holds with $\overline{L%
}=\frac{C_{d_{0}}}{f_{\ast }}L_{f}^{2}.$
\end{proposition}

\begin{proof}
For the sake of completeness we give the argument. Let $\phi :{\mathbb{R}}%
_{+}\rightarrow {\mathbb{R}}_+$ be defined by $\phi( t) = t/ \sqrt{f_*} , $
if $0 \le t \le f_*, $ and $\phi( t) = \sqrt{t} $ for $t \geq f_*.$ Then $%
\phi $ is Lipschitz with Lipschitz constant $L_\phi = \frac{1}{ \sqrt{f_*}}
. $ By (\ref{sigmaL}), the spectrum of $F_{N}(x)$ is included in $(f_{\ast
},\infty )$ so that $\sigma _{N}(x)=\phi (F_{N}(x)).$ Then the first
inequality above follows from Farforovskaya's result concerning Lipschitz
continuous functions of self adjoint operators with finite spectrum, see 
\cite{Farf}.
\end{proof}

\label{sec:42}

We discuss now $A(v)$. Recall that $\Vert A\Vert _{HS}$ denotes the
Hilbert-Schmidt norm of a $d_{0}\times d_{0}-$matrix $A.$

\begin{proposition}
\label{prop:56} Let $X=(X^{1},\ldots ,X^{N})\in L_{s}^{2}({\mathbb{R}}%
^{d\times N}).$ Then we have that 
\begin{equation*}
\Vert \widehat{S}_{s,t}^{N}(X,f)-\overline{S}_{s,t}^{N}(X,f)\Vert
_{2}^{2}\leq \frac{C_{d_{0}}}{f_{\ast }}(t-s)E\Vert F_{N}(X)-\bar{F}%
_{s}^{W}(X)\Vert _{HS}^{2}.
\end{equation*}%
If moreover the $X^{1},\ldots ,X^{N}$ are conditionally independent,
conditionally on $\mathcal{F}_{s}^{W},$ then 
\begin{equation}
E\Vert F_{N}(X)-\bar{F}_{s}^{W}(X)\Vert _{HS}^{2}\leq 4d_{0}(f^{\ast
})^{2}N^{-1}
\end{equation}%
and 
\begin{equation}
\Vert \widehat{S}_{s,t}^{N}(X,f)-\overline{S}_{s,t}^{N}(X,f)\Vert
_{2}^{2}\leq \frac{4d_{0}C_{d_{0}}(f^{\ast })^{2}}{f_{\ast }}(t-s)N^{-1}.
\label{eq:condind}
\end{equation}%
As a consequence, (\ref{Av}) holds with $\widehat{\varepsilon }_{N}=\frac{%
4d_{0}C_{d_{0}}(f^{\ast })^{2}}{f_{\ast }}N^{-1}.$
\end{proposition}

\begin{proof}
The first upper bound follows using the arguments of Propositions \ref%
{prop:52} and \ref{prop:53}. We now prove (\ref{eq:condind}) under the
additional condition that $X^{1},\ldots ,X^{N}$ are i.i.d., conditionally on 
$\mathcal{F}_{s}^{W}.$ It suffices to evaluate $E\Vert F_{N}(X)-\bar{F}%
_{s}^{W}(X)\Vert _{HS}^{2}.$ To do so, observe that 
\[
E\Vert F_{N}(X)-\bar{F}_{s}^{W}(X)\Vert _{HS}^{2} 
=\sum_{j,k=1}^{d_{0}}E\left( \frac{1}{N}\sum_{i=1}^{N}\left(
f^{j,k}(X^{i})-E(f^{j,k}(X^{i})|\mathcal{F}_{s}^{W})\right) \right) ^{2},
\]
which by conditional independence and exchangeability of the $X^{i}$ equals 
\begin{equation*}
\frac{1}{N}\sum_{j,k=1}^{d_{0}}E\left( \left(
f^{j,k}(X^{i})-E(f^{j,k}(X^{i})|\mathcal{F}_{s}^{W})\right) ^{2}\right) \leq 
\frac{4}{N}\sum_{j,k=1}^{d_0}E\left( \left( f^{j,k}(X^{i})\right)
^{2}\right) ,
\end{equation*}%
where we have used that $(a-b)^{2}\leq 2a^{2}+2b^{2}$ and Jensen's
inequality.

We conclude observing that 
\begin{equation*}
\sum_{j,k=1}^{d_{0}}\left( f^{j,k}(X^{i})\right) ^{2}=\Vert f\Vert
_{HS}^{2}(X^{i})\leq d_0 \Vert f\Vert _{op}^{2}(X^{i})\leq d_0 (f^{\ast
})^{2},
\end{equation*}%
by assumption on $f.$
\end{proof}

\section{Example}

\label{sec:5} In this section we discuss a specific particle system to which
we apply the previous results in the CLT setting. Our goal is to clarify the
main objects of the paper and to illustrate our results, and so we will
impose strong conditions on the coefficients to keep this section as light
as possible. First of all, to ease the reading, we only consider the case $%
d_0=1.$ For $i=1,...,N,$ we consider the system of particles taking values
in ${\mathbb{R}}$ and obeying, for any $1\leq i\leq N,$ and $0\leq s \le t <
\infty , $%
\begin{eqnarray}
\mathbf{X}_{s,t}^{N,i} &=&\mathbf{X}^{N,i}+\int_{s}^{t}\sigma _{0}(\mathbf{X}%
_{s,r}^{N,i},\mathbf{S}_{s,r}^{N}(\mathbf{X)},\mu _{N}(\mathbf{X}%
_{s,r}^{N},\varphi ))dr  \label{EX1} \\
& & +\int_{s}^{t}\sigma _{1}(\mathbf{X}_{s,r}^{N,i},\mathbf{S}_{s,r}^{N}(%
\mathbf{X)} , \mu _{N}(\mathbf{X}_{s,r}^{N},\varphi ))dB_{r}^{i}  \notag \\
&&
+\int_{]s,t]}\int_{E_{1}}c(u,\mathbf{X}_{s,r}^{N,i},\mathbf{S}_{s,r}^{N}(%
\mathbf{X)} , \mu _{N}(\mathbf{X}_{s,r}^{N},\varphi ))N^{i}(dr,du) ,  \notag
\end{eqnarray}
with 
\begin{eqnarray*}
\mathbf{S}_{s,t}^{N}(\mathbf{X)} &:= &\frac{1}{\sqrt{N}}\left(
\sum_{j=1}^{N}\int_{s}^{t}\overline{\sigma }(\mathbf{X}_{s,r}^{N,j})d%
\overline{B}_{r}^{j}+\sum_{j=1}^{N}\int_{]s,t]}\int_{E_{2}}\overline{c}(u,%
\mathbf{X}_{s,r}^{N,j}))\widetilde{M}^{j}(dr,du)\right) ,  \notag \\
\mu _{N}(\mathbf{X}_{s,t}^{N},\varphi ) &:= &\frac{1}{N}\sum_{j=1}^{N}%
\varphi (\mathbf{X}_{s, t}^{N,j}) .  \notag
\end{eqnarray*}%
Here, $B^{i},\overline{B}^{i},i=1,\ldots ,N,$ are independent standard
Brownian motions. Moreover, $N^{i}$ and $M^{i},i=1,\ldots ,N,$ are
independent Poisson random measures on some measurable space $(E_{1},\nu
_{1})$ ($(E_{2},\nu _{2}),$ respectively), with intensity $dr\nu _{1}(du)$ ($%
dr\nu _{2}(du),$ respectively). $\nu _{1}$ and $\nu _{2}$ are sigma-finite
measures on $E_{1}$ ($E_{2},$ respectively). For any $1 \le i \le N, $ the
measure $\widetilde{M}^{i}(dr,du)=M^{i}(dr,du)-dr\nu _{2}(du)$ is the
compensated Poisson random measure, and $\mu _{N}(\mathbf{X}_{r}^{N},dx)=%
\frac{1}{N}\sum_{i=1}^{N}\delta _{\mathbf{X}_{r}^{N,i}}(dx)$ is the
empirical measure of the system

The coefficients of the system are 
\begin{eqnarray*}
\sigma _{j} &:&{\mathbb{R}}^{2}\times {\mathbb{R}}^{m}\rightarrow {\mathbb{R}%
},j=0,1,\quad \overline{\sigma }:{\mathbb{R}}\rightarrow {\mathbb{R}}, \\
c &:&E_{1}\times {\mathbb{R}}^{2}\times {\mathbb{R}}^{m}\rightarrow {\mathbb{%
R}},\quad \overline{c}:E_{2}\times {\mathbb{R}}\rightarrow {\mathbb{R}}.
\end{eqnarray*}

We assume that all our coefficients are bounded, that $\varphi :{\mathbb{R}}%
\rightarrow {\mathbb{R}}^{m}$ is moreover Lipschitz, and we impose the usual
Lipschitz continuity conditions on the other coefficients, i.e.:%
\begin{equation}
\sum_{j=0}^{1}\left\vert \sigma _{j}(x,s, a)-\sigma _{j}(\overline{x},\bar
s, \overline{a})\right\vert +\left\vert \overline{\sigma }(x)-\overline{%
\sigma }(\overline{x})\right\vert \leq L_{\sigma }(\left\vert x-\overline{x}%
\right\vert + | s - \bar s | + \left\vert a-\overline{a}\right\vert ),
\label{EX2a}
\end{equation}%
and, for convenient positive measurable functions $\bar{c}_{1}$ and $\bar{c}%
_{2},$ 
\begin{eqnarray}
\left\vert c(u,x,s, a)-c(u,\overline{x},\bar s, \overline{a})\right\vert
&\leq &\bar{c}_{1}(u)(\left\vert x-\overline{x}\right\vert +|s - \bar s | +
\left\vert a-\overline{a}\right\vert ),  \label{EX2b} \\
\left\vert \overline{c}(u,x)-\overline{c}(u,\overline{x})\right\vert &\leq &%
\overline{c}_{2}(u)\left\vert x-\overline{x}\right\vert ,  \label{EX2bb}
\end{eqnarray}
where 
\begin{equation}
\int_{E_{1}}(\overline{c}_{1}(u)+\overline{c}_{1}^{2}(u))d\nu
_{1}(u)+\int_{E_{2}}( \overline{c}_2 (u ) + \overline{c}_{2}^{2}(u) ) d\nu
_{2}(u)<\infty .  \label{EX2c}
\end{equation}

We introduce now the associated limit equation. We define%
\begin{equation}
f(x)=\bar{\sigma}^{2}(x)+\int_{E_{2}}\bar{c}^{2}(u,x)d\nu _{2}(u),
\label{EX5}
\end{equation}%
and we assume that for all $x,$ 
\begin{equation*}
0<f_{\ast }\leq f(x)\leq f^{\ast } .
\end{equation*}%
Notice that as a consequence of (\ref{EX2a}) and (\ref{EX2bb}), $f$ is
Lipschitz continuous.

The associated limit equation is given by 
\begin{eqnarray*}
\widehat{\mathbf{X}}_{s,t}^{N,i} &=&\mathbf{X}^{i}+\int_{s}^{t}\sigma _{0}(%
\widehat{\mathbf{X}}_{s,r}^{N,i},\widehat{\mathbf{S}}_{s,r}^{N},\bar{\mu}%
_{s,r}(\varphi ))dr+\int_{s}^{t}\sigma _{1}(\widehat{\mathbf{X}}_{s,r}^{N,i},%
\widehat{\mathbf{S}}_{s,r}^{N},\bar{\mu}_{s,r}(\varphi ))dB_{r}^{i} \\
&&+\int_{]s,t]}\int_{E_{1}}c(u,\widehat{\mathbf{X}}_{s,r}^{N,i},\widehat{%
\mathbf{S}}_{s,r}^{N},\bar{\mu}_{s,r}(\varphi ))N^{i}(dr,du)
\end{eqnarray*}%
with 
\begin{eqnarray*}
\bar{\mu}_{s,r} &=&{\mathcal{L}}(\widehat{\mathbf{X}}_{s,r}^{N,i}\mid 
\mathcal{F}_{r}^{W})\quad \mbox{and}\quad \\
\widehat{\mathbf{S}}_{s,r}^{N} &=&\int_{s}^{t}\sqrt{\bar{\mu}_{s,r}(f)}%
dW_{r}=\int_{s}^{t}\sqrt{E(f(\widehat{\mathbf{X}}_{s,r}^{N,i})\mid \mathcal{F%
}_{r}^{W})}dW_{r}
\end{eqnarray*}%
for any $i\geq 1,$ and $s \le r\leq t.$

In the above equation, we suppose that $\mathbf{X}^{i},i\geq 1,$ is an
i.i.d. sequence of random variables belonging to $L_{s}^{2},$ independent of 
$W_{s},$ and distributed according to (the deterministic measure) $\bar{\mu}%
_{s}.$

Notice that, because of the presence of the common noise $dW_{r},$ the
components $\widehat{\mathbf{X}}_{s,t}^{N,i},i=1,...,N,$ are identically
distributed, but they are not independent -- they are in fact only
independent conditionally on $W,$ and the above equation is an example of a 
\textit{conditional McKean-Vlasov} equation.

Under the hypotheses (\ref{EX2a}-\ref{EX2c}), standard arguments imply the
trajectorial uniqueness of (\ref{EX3}) (see e.g. \cite{carl} and \cite{ELL}%
). However, the presence of the conditional expectations makes the proof of
existence of the solution more delicate. In what follows, we propose to give
the proof of the existence by applying the results of the previous sections.
We will also automatically obtain an estimate of the $\mathcal{W}_{2}-$%
distance between the two systems.

\begin{theorem}
\label{Maine2} \textbf{A }Grant the assumptions (\ref{EX2a})--(\ref{EX2bb}),
and suppose moreover that $\varphi $ is Lipschitz and bounded and that $%
0<f_{\ast }\leq f(x)\leq f^{\ast }<\infty $ for every $x\in {\mathbb{R}}$.
Let $\mathbf{X}^{i},i\geq 1,$ be i.i.d., belonging to $L_{s}^{2},$
independent of $W_{s},$ and distributed according to $\bar{\mu}_{s}.$ Then
there exists a unique solution $(\widehat{\mathbf{X}}_{s,t})_{t\geq s}$ of
the system of equations (\ref{EX3}).

\noindent \textbf{B }For every $T>0,$ 
\begin{equation*}
\sup_{s \le t\leq T}\mathcal{W}_{2}(\mathbf{X}_{s,t}^{N}(X),\widehat{\mathbf{%
X}}_{s,t}(X))\leq \frac{C_{T}}{\sqrt{N}},
\end{equation*}%
where $C_{T}$ is a constant depending only on the model parameters, on $%
C_{Bonis}$ and on $T.$
\end{theorem}

The reminder of this section is devoted to the proof of this result. In a
first step we reformulate this example in the framework of Sections \ref%
{sec:4} and \ref{sec:CLT}. In a second step, we verify the hypotheses needed
to obtain Theorem \ref{theo:main}. And in a last step, we check the sewing
property to be able to conclude with item \textbf{B} of Theorem \ref%
{theo:main}.

\subsection{Coming back to the framework of Sections \protect\ref{sec:4} and 
\protect\ref{sec:CLT}}

Notice that the equation (\ref{EX1}) does not define a Markov process, since
it depends on the past trajectory through the expression $\mathbf{S}%
_{s,r}^{N}(\mathbf{X)} .$ So a first step is to replace $\mathbf{X}%
_{s,t}^{N,i}$ by the couple $\mathbb{X}_{s,t}^{N,i}=(\mathbb{X}%
_{s,t}^{N,i,1},\mathbb{X}_{s,t}^{N,i,2})\in {\mathbb{R}}^{2}$ which is
solution of the equation 
\begin{eqnarray}
\mathbb{X}_{s,t}^{N,i,1} &=&\mathbb{X}^{N,i,1}+\int_{s}^{t}\sigma _{0}(%
\mathbb{X}_{s,r}^{N,i, 1},\mathbb{X}_{s,t}^{N,i,2} , \mu _{N}(\mathbb{X}%
_{s,r}^{N, 1},\varphi ))dr  \label{EX1'} \\
&& +\int_{s}^{t}\sigma _{1}(\mathbb{X}_{s,r}^{N,i, 1}, \mathbb{X}%
_{s,r}^{N,i, 2}, \mu _{N}(\mathbb{X}_{s,r}^{N, 1},\varphi ))dB_{r}^{i} 
\notag \\
&&+\int_{]s,t]}\int_{E_{1}}c(u,\mathbb{X}_{s,r}^{N,i, 1},\mu _{N}(\mathbb{X}%
_{s,r}^{N, 1 },\varphi ))N^{i}(dr,du) ,  \notag
\end{eqnarray}
where 
\begin{equation*}
\mu _{N}(\mathbb{X}_{s,t}^{N, 1},\varphi ) := \frac{1}{N}\sum_{j=1}^{N}%
\varphi (\mathbb{X}_{s, t}^{N,j, 1})
\end{equation*}
and 
\begin{eqnarray}  \label{eq:EX1suite}
\mathbb{X}_{s,t}^{N,i,2} &=&\frac{1}{\sqrt{N}}\left(
\sum_{j=1}^{N}\int_{s}^{t}\overline{\sigma }(\mathbb{X}_{s,r}^{N,j,1})d%
\overline{B}_{r}^{j}+\sum_{j=1}^{N}\int_{]s,t]}\int_{E_{2}}\overline{c}(u,%
\mathbb{X}_{s,r}^{N,j,1}))\widetilde{M}^{j}(dr,du)\right) .  \notag
\end{eqnarray}

Under our assumptions, Theorem 1.2 of Graham \cite{carl} implies that (\ref%
{EX1'})--(\ref{eq:EX1suite}) admits a unique strong non-explosive adapted
solution which is Markov. As a consequence, existence and uniqueness follows
also for the solution of (\ref{EX1}).

The solution of the limit equation will be the couple $\widehat{\mathbb{X}}%
_{s,t}^{i}=(\widehat{\mathbb{X}}_{s,t}^{i,1},\widehat{\mathbb{X}}%
_{s,t}^{i,2}),$ $0 \le s \le t < \infty, $ which verifies 
\begin{eqnarray*}
\widehat{\mathbb{X}}_{s,t}^{N,i,1} &=&\widehat{\mathbb{X}}%
^{N,i,1}+\int_{s}^{t}\sigma _{0}(\widehat{\mathbb{X}}_{s,r}^{N,i, 1 }, 
\widehat{\mathbb{X}}_{s,r}^{N,i, 2 }, \bar{\mu}^1_{s,r}(\varphi ))dr \\
&&+ \int_{s}^{t}\sigma _{1}(\widehat{\mathbb{X}}_{s,r}^{N,i, 1}, \widehat{%
\mathbb{X}}_{s,r}^{N,i, 2 }, \bar{\mu}^1_{s,r}(\varphi ))dB_{r}^{i} \\
&&+\int_{]s,t]}\int_{E_{1}}c(u,\widehat{\mathbb{X}}_{s,r}^{N,i, 1},\widehat{%
\mathbb{X}}_{s,r}^{N,i, 2} \bar{\mu}_{s,r}^1(\varphi ))N^{i}(dr,du), \\
\widehat{\mathbb{X}}_{s,t}^{N,i,2}&=& \widehat{\mathbb{X}}%
^{N,i,2}+\int_{s}^{t}\sqrt{E(f(\widehat{\mathbb{X}}_{s,r}^{N,i,1})\mid 
\mathcal{F}_{r}^{W})}dW_{r} ,
\end{eqnarray*}%
with $\bar{\mu}^1_{s,r}(\varphi )=E(\varphi (\widehat{\mathbb{X}}%
_{s,r}^{N,i,1})\mid \mathcal{F}_{r}^{W}).$

Notice that in (\ref{EX1'})--(\ref{eq:EX1suite}) we took $\widehat{\mathbb{X}%
}^{N,i,2} = 0,$ but other choices are of course possible.

With this construction, we have $\mathbf{X}_{s,t}^{N}=\mathbb{X}%
_{s,t}^{N,1}, $ $\mathbf{S}_{s,t}^{N}=\mathbb{X}_{s,t}^{N,2},$ $\widehat{%
\mathbf{X}}_{s, t}^{N}=\widehat{\mathbb{X}}_{s,t}^{N,1}$ and $\widehat{%
\mathbf{S}}_{s,t}^{N}=\mathbb{X}_{s,t}^{N,2}.$ As a consequence, in order to
prove our result, it will be sufficient to prove 
\begin{equation}
\sup_{s \le t\leq T}\mathcal{W}_{2}(\mathbf{X}_{s,t}^{N}(X),\widehat{\mathbf{%
X}}_{s,t}(X))\leq \sup_{s \le t\leq T}\mathcal{W}_{2}(\mathbb{X}_{s,t}^{N},%
\widehat{\mathbb{X}}_{s,t}^{N})\leq \frac{C_{T}}{\sqrt{N}} .  \label{EX3}
\end{equation}

In a second step we have to define the Euler schemes which approximate $%
\mathbb{X}_{s,t}^{N}$ respectively $\widehat{\mathbb{X}}_{s,t}^{N}.$
According to (\ref{pf7}) and (\ref{pf9}), the one step Euler schemes which
approximate $\mathbb{X}_{s,t}^{N}$ respectively $\widehat{\mathbb{X}}%
_{s,t}^{N}$ are given as follows. For every $i=1,...,N,$%
\begin{equation*}
\mathcal{X}_{s,t}^{N,i,1}(X)=X^{i,1}+h_{s,t}^{i, 1}(X); \quad \mathcal{X}%
_{s,t}^{N,i,2}(X)=X^{i,2}+S_{s,t}^{N}(X)
\end{equation*}%
and 
\begin{equation*}
\widehat{\mathcal{X}}_{s,t}^{N,i, 1}(X)=X^{i,1}+\widehat{h}%
_{s,t}^{i,1}(X^{i});\quad \widehat{\mathcal{X}}_{s,t}^{N,i,2}(X)=X^{i,2}+%
\widehat{S}_{s,t}(X^{i}),
\end{equation*}%
with $h_{s,t}^{i, 1},\widehat{h}_{s,t}^{i, 1}$ and $S_{s,t}^{N}$ and $%
\widehat{S}_{s,t}^{N}$ to be defined in the following section.

\subsection{The function $H$}

In this example, $H_{s,t}^{i}:{\mathbb{R}}^{2}\times {\mathbb{R}}%
^{m}\rightarrow {\mathbb{R}}^{2}$ is given by 
\begin{eqnarray}
H_{s,t}^{i,1}(x,a) &=&\sigma _{0}(x^1, x^2 ,a)(t-s)+\sigma _{1}(x^1, x^2,
a)(B_{t}^{i}-B_{s}^{i}) \\
&& +\int_{]s,t]}\int_{E_{1}}c(u,x^1, x^2 ,a)N^{i}(dr,du),  \notag \\
H_{s,t}^{i,2}(x,a) &=&0 .  \label{EX4}
\end{eqnarray}%
Clearly, 
\begin{equation*}
E_{s}(H_{s,t}^{i,1}(x,a))=\sigma _{0}(x^1, x^2,a)(t-s)+\int_{E_{1}}c(u,x^1,
x^2 ,a)d\nu _{1}(u)(t-s),
\end{equation*}%
such that hypothesis $(H,i)$ is satisfied. Moreover, due to (\ref{EX2a}) and
(\ref{EX2b}), writing $x = (x^1, x^2), y = (y^1, y^2 ),$ 
\begin{equation*}
\left\vert E_{s}(H_{s,t}^{i}(x,a))-E_{s}(H_{s,t}^{i}(y,b))\right\vert \leq
(t-s)(L_{\sigma }+\int_{E_{1}}\overline{c}_{1}(u)d\nu _{1}(u))(\left\vert
x-y\right\vert +\left\vert a-b\right\vert ),
\end{equation*}%
implying (\ref{pf1}).

Concerning the martingale part, we have that 
\begin{equation*}
\widetilde{E}_{s}(H_{s,t}^{i,1}(X^{i},a))=\sigma
_{1}(X^{i},a)(B_{t}^{i}-B_{s}^{i})+\int_{]s,t]}\int_{E_{1}}c(u,X^{i},a)%
\widetilde{N}^{i}(dr,du).
\end{equation*}%
Using independence and the isometry property, we get 
\begin{multline*}
E\left\vert \widetilde{E}_{s}(H_{s,t}^{i}(X^{i},a))-\widetilde{E}%
_{s}(H_{s,t}^{i}(Y^{i},b))\right\vert ^{2} \\
=(t-s)\left( \left\vert \sigma _{1}(X^{i},a)-\sigma _{1}(Y^{i},b)\right\vert
^{2}+\int_{E_{1}}\left\vert c(u,X^{i},a)-c(u,Y^{i},b)\right\vert ^{2}d\nu
_{1}(u)\right) \\
\leq (t-s)(L_{\sigma_1^{2}}+\int_{E_{1}}\overline{c}_{1}^{2}(u)d\nu
_{1}(u))(\left\vert x-y\right\vert ^{2}+\left\vert a-b\right\vert ^{2}),
\end{multline*}%
implying that also (\ref{pf1'}) holds.

\subsection{The perturbation term $S^N_{s, t}$}

Here we have 
\begin{equation*}
v_{s,t}^{i}(x^{i},\omega )=\overline{\sigma }(x^{i,1})(\overline{B}_{t}^{i}-%
\overline{B}_{s}^{i})+\int_{]s,t]}\int_{E_{2}}\overline{c}(u,x^{i, 1})%
\widetilde{M}^{i}(dr,du).
\end{equation*}%
We define $S_{s,t}^{N}(x,f)=\frac{1}{\sqrt{N}}%
\sum_{i=1}^{N}v_{s,t}^{i}(x^{i},\omega ).$  $\overline{%
\sigma }$ and $\overline{c}^{i}$\ verify (\ref{EX2a}) and (\ref{EX2bb}) and $%
E\left\vert v_{s,t}^{i}(x^{i},\omega )\right\vert ^{2}=f(x^{i,1})(t-s),$
where $f$ is defined in (\ref{EX5}). Our assumptions are stated such that
condition $\mathbf{A(f)}$ is satisfied.

As a consequence, item \textbf{A} of Theorem \ref{theo:main} follows.

\subsection{Checking the sewing property}

Using the same notations as in Section \ref{sec:4}, we finally introduce 
\begin{equation*}
\widehat{S}_{s,t}(X^{i})=\sqrt{E(f(X^{i, 1 })|\mathcal{F}_{s}^{W})}%
(W_{t}-W_{s}),
\end{equation*}%
and we have, for each fixed $i$ 
\begin{equation*}
\widehat{\mathcal{X}}_{s,t}^{i,1}(X)=X^{i,1}+\widehat{h}_{s,t}^{i,
1}(X^{i}), \quad \widehat{\mathcal{X}}_{s,t}^{i,2}(X)=X^{i, 2}+\widehat{S}%
_{s,t}^{N}(X^{i}).
\end{equation*}%
Our aim now is to prove that, for a fixed $i,$ $\widehat{\mathcal{X}}%
_{s,t}^{i}(X)$ verifies the sewing property. Notice that $\widehat{\mathcal{X%
}}_{s,t}^{i}(X)$ does only depend on $X^{i},$ and therefore, in what
follows, we drop the index $i$ and simply write $X$ instead of $X^{i}$ which
now stands for a fixed particle taking values in ${\mathbb{R}}^2$ (and not
for the whole particle system).

We introduce the equation which is associated to the one step Euler scheme. 
For this purpose it is convenient to fix a c\`{a}dl\`{a}g function $\eta
:(s,\infty )\rightarrow (s,\infty )$ which is non-decreasing, takes only a
finite number of values and satisfies $s\leq \eta (t)\leq t$ for every $%
t\geq s.$ To any such function $\eta ,$ we associate the solution $(\mathcal{%
X}_{s,t}^{\eta }(X))_{t\geq s}$ of the discretized equation 
\begin{multline}
\mathcal{X}_{s,t}^{\eta ,1}(X)=X^{1}+\int_{s}^{t}\sigma _{0}(\mathcal{X}%
_{s,\eta (r)}^{\eta }(X),E(\varphi (\mathcal{X}_{s,\eta (r)}^{\eta
,1}(X))\mid \mathcal{F}_{r}^{W}))dr  \label{eq:Xeta} \\
+\int_{s}^{t}\sigma _{1}(\mathcal{X}_{s,\eta (r)}^{\eta }(X),E(\varphi (%
\mathcal{X}_{s,\eta (r)}^{\eta ,1}(X))\mid \mathcal{F}_{r}^{W}))dB_{r} \\
+\int_{]s,t]}\int_{E_{1}}c(u,\mathcal{X}_{s,\eta (r)}^{\eta }(X),E(\varphi (%
\mathcal{X}_{s,\eta (r)}^{\eta ,1}(X))\mid \mathcal{F}_{r}^{W}))N(dr,du) ,
\end{multline}
together with 
\begin{equation*}
\mathcal{X}_{s,t}^{\eta ,2}(X)=X^{2}+\int_{s}^{t}\sqrt{E(f(\mathcal{X}%
_{s,\eta (r)}^{\eta ,1}(X))\mid \mathcal{F}_{r}^{W})}dW_{r} ,
\end{equation*}
where $X\in L_{s}^{2}.$ Since $\eta $ takes only a finite number of values,
the above equation has a unique solution which can be constructed explicitly
step by step. We also denote by $\mathcal{X}_{s,t}^{0}(X)$ the solution of
the above equation for the choice of the constant function $\eta (t)=s,$ for
every $t\geq s.$ In this case%
\begin{equation*}
E(\varphi (\mathcal{X}_{s,\eta (r)}^{\eta ,1}(X))\mid \mathcal{F}%
_{r}^{W}))=E(\varphi (X^{1, 1})\mid \mathcal{F}_{s}^{W}),
\end{equation*}%
such that the corresponding equation reads as 
\begin{eqnarray*}
\mathcal{X}_{s,t}^{0,1}(X) &=&X^{1}+\int_{s}^{t}\sigma _{0}(X,E(\varphi
(X^{1})\mid \mathcal{F}_{s}^{W}))dr+\int_{s}^{t}\sigma _{1}(X,E(\varphi
(X^{1})\mid \mathcal{F}_{s}^{W}))dB_{r} \\
&+&\int_{]s,t]}\int_{E_{1}}c(u,X,E(\varphi (X^{1})\mid \mathcal{F}%
_{s}^{W}))N(dr,du) , \\
\mathcal{X}_{s,t}^{0,2}(X) &=&X^{2}+\int_{s}^{t}\sqrt{E(f(X^{1}))\mid 
\mathcal{F}_{s}^{W})}dW_{r},
\end{eqnarray*}%
which is the one step Euler scheme associated to the equation, that is, $%
\mathcal{X}_{s,t}^{0}=\widehat{\mathcal{X}}_{s,t}.$

\begin{lemma}
Grant Assumptions (\ref{EX2a})--(\ref{EX5}) and suppose that $X \in L^2 _s.$
Then we have 
\begin{eqnarray}
E(\left\vert E_{s}((\mathcal{X}_{s,t}^{\eta }(X))-\mathcal{X}%
_{s,t}^{0}(X))\right\vert ^{2}) &\leq &C(1+E(\left\vert X\right\vert
^{2}))(t-s)^{3} ,  \label{EX6} \\
E(\left\vert \widetilde{E}_s(\mathcal{X}_{s,t}^{\eta }(X)-\mathcal{X}%
_{s,t}^{0}(X))\right\vert ^{2}) &\leq &C(1+E(\left\vert X\right\vert
^{2}))(t-s)^{2} ,  \label{EX7}
\end{eqnarray}%
and 
\begin{equation}
E(\left\vert \mathcal{X}_{s,t}^{0}(X)-\mathcal{X}_{s,t}^{0}(Y)\right\vert
^{2})\leq CE(\left\vert X-Y\right\vert ^{2})(t-s) .  \label{EX8}
\end{equation}
\end{lemma}

The proof is standard and we skip it.

We now fix $s<r<t$ and define 
\begin{eqnarray*}
\eta (u) &=&r\qquad \mbox{if } r\leq u, \\
&=&s\qquad \mbox{if } s\leq u < r.
\end{eqnarray*}%
We see that 
\begin{equation*}
\mathcal{X}_{s,t}^{\eta }=\widehat{\mathcal{X}}_{r,t}(\widehat{\mathcal{X}}%
_{s,r}(X)),
\end{equation*}%
and, recalling that $\widehat{\mathcal{X}}_{s,t}=\mathcal{X}_{s,t}^{0},$
also that 
\begin{eqnarray*}
E\left\vert E_{s}(\widehat{\mathcal{X}}_{s,t}(X)-\widehat{\mathcal{X}}_{r,t}(%
\widehat{\mathcal{X}}_{s,r}(X)))\right\vert ^{2} &=&E\left\vert E_{s}(%
\widehat{\mathcal{X}}_{s,t}(X)-\mathcal{X}_{s,t}^{\eta }(X))\right\vert
^{2}\leq C(t-s)^{3} , \\
E\left\vert \widetilde{E}_{s}(\widehat{\mathcal{X}}_{s,t}(X)-\widehat{%
\mathcal{X}}_{r,t}(\widehat{\mathcal{X}}_{s,r}(X))) \right\vert ^{2}
&=&E\left\vert \widetilde{E}_{s}(\widehat{\mathcal{X}}_{s,t}(X)-\mathcal{X}%
_{s,t}^{\eta }(X))\right\vert ^{2}\leq C(t-s)^{2},
\end{eqnarray*}
which shows that the sewing property holds with $\beta =\frac{1}{2}.$

It is also standard to prove that%
\begin{eqnarray*}
E\left\vert E_{s}(\widehat{\mathcal{X}}_{s,t}(X)-\widehat{\mathcal{X}}%
_{s,t}(Y))\right\vert ^{2} &\leq &C(t-s)^{2}E\left\vert X-Y\right\vert ^{2},
\\
E\left\vert \widetilde{E}_{s}(\widehat{\mathcal{X}}_{s,t}(X)-\widehat{%
\mathcal{X}}_{s,t}(Y))\right\vert ^{2} &\leq &C(t-s)E\left\vert
X-Y\right\vert^2,
\end{eqnarray*}%
such that the Lipschitz property is also satisfied.

Summarizing the above discussion, we obtain the following

\begin{proposition}
Grant Assumptions (\ref{EX2a})--(\ref{EX5}). Then $\widehat{\mathcal{X}}%
_{s,t}$ verifies the assumption $\mathbf{A(}\beta )$ of (\ref{Abeta}) with $%
\beta =1/2,$ and its associated flow solution $\widehat{\mathbb{X}}$ is the
unique solution of (\ref{EX3}).
\end{proposition}

\begin{proof}
Recall the definition of the dyadic partition $\pi_n ( s, t) $ in (\ref%
{eq:Xn}) and that $t_k = k 2 ^{ - n }$ for any $k.$ So we have 
\begin{equation*}
\widehat{X}_{s,t}^{n}=\widehat{\mathcal{X}}_{t_{k (t) },t} \circ \widehat{%
\mathcal{X}}_{t_{k(t) - 1 } , t_{k(t)} } \circ ...\circ \widehat{\mathcal{X}}%
_{s ,t_{k(s)}+1} .
\end{equation*}%
As in (\ref{pf13bis}), the above scheme determines the values of $\widehat{X}%
^n_{s , t_k} $ for all dyadic points $t_k \in ( s, t) . $

We now consider the function $\eta $ which is defined on $[s,t]$ by $\eta
(r)=t_{i}$ for all $r$ with $t_{i}\leq r<t_{i+1},$ and $\eta (r)=t$ for all $%
r\geq t.$ Then $\widehat{X}_{s,t_{k}}^{n}=\mathcal{X}_{s,t_{k}}^{\eta }(X)$
for all $k$ such that $t_{k}\in (s,t),$ and moreover $\widehat{X}_{s,t}^{n}=%
\mathcal{X}_{s,t}^{\eta }(X).$ In other words, $(\mathcal{X}_{s,r}^{\eta
}(X))_{r\geq s}$ is a continuous interpolation of $\widehat{X}_{s,t_{k}}^{n}$
in between the dyadic points. We rewrite equation (\ref{eq:Xeta}) as 
\begin{multline}
\mathcal{X}_{s,t}^{\eta ,1}(X)=X^1+\int_{s}^{t}\sigma _{0}(\mathcal{X}%
_{s,r}^{\eta }(X),E(\varphi (\mathcal{X}_{s,r}^{\eta ,1}(X))\mid \mathcal{F}%
_{r}^{W}))dr  \label{eq:Xetabis} \\
+\int_{s}^{t}\sigma _{1}(\mathcal{X}_{s,r}^{\eta }(X),E(\varphi (\mathcal{X}%
_{s,r}^{\eta ,1}(X))\mid \mathcal{F}_{r}^{W}))dB_{r} \\
+\int_{]s,t]}\int_{E_{1}}c(u,\mathcal{X}_{s,r}^{\eta }(X),E(\varphi (%
\mathcal{X}_{s,r}^{\eta ,1}(X))\mid \mathcal{F}%
_{r}^{W}))N(dr,du)+I_{s,t}^{0}+I_{s,t}^{1}+J_{s,t}
\end{multline}
together with 
\begin{equation*}
\mathcal{X}_{s,t}^{\eta ,2}(X)=X^{2}+\int_{s}^{t}\sqrt{E(f(\mathcal{X}%
_{s,r}^{\eta ,1}(X))\mid \mathcal{F}_{r}^{W})}dW_{r}+I_{s,t}^{W},
\end{equation*}
where 
\begin{eqnarray*}
I_{s,t}^{0} &=&\int_{s}^{t}\sigma _{0}(\mathcal{X}_{s,\eta (r)}^{\eta
},E(\varphi (\mathcal{X}_{s,\eta (r)}^{\eta ,1}(X))\mid \mathcal{F}%
_{r}^{W}))-\sigma _{0}(\mathcal{X}_{s,r}^{\eta },E(\varphi (\mathcal{X}%
_{s,r}^{\eta ,1}(X))\mid \mathcal{F}_{r}^{W}))dr, \\
I_{s,t}^{1} &=&\int_{s}^{t}\sigma _{1}(\mathcal{X}_{s,\eta (r)}^{\eta
,1},E(\varphi (\mathcal{X}_{s,\eta (r)}^{\eta }(X))\mid \mathcal{F}%
_{r}^{W}))-\sigma _{1}(\mathcal{X}_{s,r}^{\eta },E(\varphi (\mathcal{X}%
_{s,r}^{\eta ,1}(X))\mid \mathcal{F}_{r}^{W}))dB_{r}, \\
J_{s,t} &=&\int_{]s,t]}\int_{E_{1}}c(u,\mathcal{X}_{s,\eta (r)}^{\eta
},E(\varphi (\mathcal{X}_{s,\eta (r)}^{\eta ,1}(X))\mid \mathcal{F}_{r}^{W}))
\\
&&\quad \quad \quad \quad \quad \quad \quad \quad \quad \quad \quad \quad
-c(u,\mathcal{X}_{s,r}^{\eta },E(\varphi (\mathcal{X}_{s,r}^{\eta
,1}(X))\mid \mathcal{F}_{r}^{W}))N(dr,du), \\
I_{s,t}^{W} &=&\int_{s}^{t}\sqrt{E(f(\mathcal{X}_{s,\eta (r)}^{\eta ,1})\mid 
\mathcal{F}_{\tau (r)}^{W})}-\sqrt{E(f(\mathcal{X}_{s,r}^{\eta ,1})\mid 
\mathcal{F}_{r}^{W})}dW_{r}.
\end{eqnarray*}%
Standard computations give 
\begin{equation*}
E\left\vert I_{s,t}^{0}\right\vert ^{2}+E\left\vert I_{s,t}^{1}\right\vert
^{2}+E\left\vert J_{s,t}\right\vert ^{2}+E\left\vert I_{s,t}^{W}\right\vert
^{2}\rightarrow 0,
\end{equation*}%
as $n\rightarrow \infty .$ Since $\lim_{n\rightarrow \infty }E\left\vert 
\widehat{X}_{s,t}^{n}-\widehat{\mathbb{X}}_{s,t}\right\vert ^{2}=0,$ passing
to the limit in equation (\ref{eq:Xeta}), we obtain (\ref{EX3}). The proof
of uniqueness is standard.
\end{proof}

We close this section with the

\begin{proof}[Proof of Theorem \protect\ref{Maine2}]
It is standard to show that the one step Euler scheme $\mathcal{X}_{s,t}^{N}$
associated to the particle system (\ref{EX1'}) satisfies the sewing property
and that the flow solution associated to $\mathcal{X}_{s,t}^{N}$ is indeed
given by the unique solution of system (\ref{EX1'}). The claim of Theorem %
\ref{Maine2} follows then from Theorem \ref{theo:main}.
\end{proof}

\bigskip

\section{Multi layer particle systems}

In this section we present an example which is a variant of the previous
model structured within several layers. To ease the reading we stick to the
one-dimensional case.

As before, we fix $s < t $ and we consider the dyadic partition $\pi_n ( s,
t) = \{ s = t_0 < t_1 < \ldots < t_l = t\}.$ We introduce some notation. Let 
$r_{0} \geq 1 $ be fixed. To encode the different particles belonging to
different layers, we make use of multi-indices $(l, i), $ where $1 \le l \le
r_0 + 1 $ denotes the layer and $1 \le i \le N $ the number of the particle
within a given layer. Each layer is constituted of $N$ particles, and we
consider $r_0 + 1 $ layers.

Our construction is as follows. For each $1 \le r \le r_0 $ and each $\delta
>0,$ we introduce random functions $v_{\delta }^{ (r, i ) }:\Omega \times {%
\mathbb{R}} \rightarrow {\mathbb{R}}$ such that the following properties
hold.

\begin{enumerate}
\item For all $1 \le r \le r_0 $ and for all $x$ we have that $v_\delta^{r,
i } ( \omega, x ) , 1 \le i \le N , $ are all independent and centered.

\item For all $r^{\prime }< r, $ and for all $x, x^{\prime }, $ the $%
v_\delta^{ r^{\prime }, i^{\prime }}( \omega, x) $ and $v_\delta^{r, i }
(\omega, x^{\prime }) $ are all independent, for any choice of $i, i^{\prime
}.$

\item We have that 
\begin{equation*}
E(v_{\delta }^{r, i } (\omega ,x)^{2})=f_{r }(x)\delta , 
\end{equation*}
for some given functions $f_{r}: {\mathbb{R}} \to {\mathbb{R}}_+
,r=1,...,r_{0},$ which are lower-bounded, bounded and Lipschitz continuous.
\end{enumerate}

Let us introduce some further notation. To the last layer, having number $%
r_0 + 1, $ we associate outcome variables $X_{t_k}^{ i } ,1 \le i \le N , k
= 0, \ldots, l. $ Moreover, to each intermediate layer $1, \ldots, r_0,$ we
associate auxiliary processes $Y_{t_k}^{ r, i^{\prime }}, k=0, \ldots, l, $
for $r \le r_0 $ and $1 \le i^{\prime }\le N ,$ representing intermediate
results.

We initialize our scheme by taking $N$ i.i.d. copies per layer,
independently for any layer. In other words, for any $1 \le r \le r_0, $ we
take $Y_{t_0}^{ r, i }, 1 \le i \le N, $ i.i.d., and for layer $r_0 + 1, $
we take $X_{t_0}^{ i }, 1 \le i \le N, $ i.i.d. as well. Variables
associated to different layers are independent.

The scheme then proceeds recursively as follows. Suppose that all variables $%
X_{t_k}^{ i }, 1 \le i \le N , $ and $Y_{t_k}^{r, i }, r \in
\{1,...,r_{0}\}, 1 \le i \le N $\ are given. Then we construct the variables 
$Y_{t_{k+1}}^{r, i }, r =r_0, \ldots, 1 $ for all $1 \le i \le N, $ and
finally $X_{t_{k+1}}^{ i }, 1 \le i \le N,$  in the following way.

\bigskip \noindent \textbf{Step 1.} For any $1 \le i \le N, $ we put 
\begin{equation*}
Y_{t_{k+1}}^{r_0, i }=Y_{t_k}^{r_0, i }+\frac{1}{\sqrt{N}}%
\sum_{j=1}^{N}v_{\gamma_k} ^{r_0 ,j}(\omega ,X_{t_k}^{ j}).
\end{equation*}

\noindent\textbf{Step 2.} We define recursively for any $r = r_0- 1, \ldots,
1$ and for any $1 \le i \le N, $ 
\begin{equation*}
Y_{t_{k+1}}^{r, i }=Y_{t_k}^{r, i }+\frac{1}{\sqrt{N}}\sum_{j=1}^{N}v_{%
\gamma _{k}}^{r ,j}(\omega ,Y_{t_{k+1}}^{r+1 ,j}).
\end{equation*}%
Notice that the construction of $Y_{t_{k+1}}^{r, i } , 1 \le i \le N, $
involves only random variables of the next layer $r+1.$

\noindent\textbf{Step 3.} Finally, as we already know $Y_{t_{k+1}}^{1, j}$
for every $1 \le j \le N ,$ we may now close the loop, and we put, for any $%
1 \le i \le N, $ 
\begin{equation*}
X_{t_{k+1}}^{ i }=X_{t_k}^{i }+b(Y_{t_{k+1}}^{1, i} )\gamma_{k}.
\end{equation*}

We discuss now the structure of the limit system in discrete time. Due to
the CLT, there will be $r_0$ independent one-dimensional Brownian motions $%
W^1, \ldots, W^{r_0} .$ Each Brownian motion $W^r $ describes the transition
from layer $r+1 $ to layer $r$ such that each particle belonging to layer $r$
will be influenced by $W^r.$ We write $\mathcal{W}^r_s = \sigma \{ W_u^r, u
\le s \}.$

Once more, we initialize our scheme by taking i.i.d. copies within each
layer, such that particles belonging to different layers are also
independent.

The scheme works then recursively as follows. Suppose that $\hat X_{t_k}^i ,
1 \le i \le N, $ and $\hat Y_{t_k}^{r, i }, 1 \le r \le r_0, 1 \le i \le N , 
$ are given.

\bigskip \noindent \textbf{Step 1.} For any $1 \le i \le N, $ we put 
\begin{equation*}
\widehat{Y}_{t_{k+1}}^{r_0, i }=\widehat{Y}_{t_k}^{r_0, i }+\sqrt{%
E(f_{r_{0}}(\widehat{X}_{t_k}^{1}) | \mathcal{W}_{t_k}^1)}\times
(W_{t_{k+1}}^{ r_{0}}-W_{t_{k}}^{r_{0}}).
\end{equation*}

\noindent\textbf{Step 2.} For $1\leq r\leq r_{0}-1$ and for any $1 \le i \le
N, $ we define 
\begin{equation*}
\widehat{Y}_{t_{k+1}}^{r,i }=\widehat{Y}_{t_k}^{r,i }+\sqrt{E(f_{r}(\widehat{%
Y}_{t_{k+1}}^{r+1, 1})| \mathcal{W}_{t_{k+1}}^{r+1})} \times (W_{t_{k+1}}^{r
}-W_{t_{k}}^{r}).
\end{equation*}

\noindent\textbf{Step 3.} Finally for $1 \le i \le N, $ we define%
\begin{equation*}
\widehat{X}_{t_{k+1}}^{i }=\widehat{X}_{t_k}^{i }+b(\hat Y_{t_{k+1}}^{1, i})
\gamma_{k}.
\end{equation*}

\begin{remark}
Notice that the associated limit equation in continuous time reads as 
\begin{equation*}
\widehat{\mathbb{X}}_{0, t}\hat{A}~= \widehat{\mathbb{X}}_{0}+\int_{0}^{t}b(%
\int_{0}^{t_{1}}dW_{t_{2}}^{1}\Phi
_{1}(\int_{0}^{t_{2}}dW_{t_{2}}^{2}...\Phi
_{r_{0}-1}(\int_{0}^{t_{r_0}}dW_{t_{r_{0}}}^{r_{0}}\Phi _{r_{0}}(\widehat{%
\mathbb{X}}_{t_{r_{0}}}^{ })...)dt_{1} ,
\end{equation*}
where 
\begin{equation*}
\Phi _{r}(Z):=\sqrt{E(f_{r}(Z) | \mathcal{W}^{r+1} )},
\end{equation*}
and where we identify $r_0+1 $ with $1.$
\end{remark}

A straightforward adaptation of the arguments of Section \ref{sec:CLT} then
implies the following

\begin{theorem}
For any $1 \le i \le N, $ and for all $k, $ we have that 
\begin{equation*}
W_{2}(X_{t_{k}}^{i},\widehat{X}_{t_{k}}^{i })\leq C_T N^{-1/2}.
\end{equation*}
\end{theorem}

We skip the proof since it follows similarly to the proof of Theorem \ref%
{theo:main}, taking into account (\ref{MD4}) and Proposition \ref{prop:56}.

\end{document}